\documentclass[12pt]{article}

\input{mathdefs.sty}

\usepackage[latin1]{inputenc}   
\usepackage[T1]{fontenc}       
\usepackage{amsmath}
\usepackage{amssymb}
\usepackage{amsthm}

\usepackage[usenames,dvipsnames]{color}

\usepackage{bm}
\usepackage{mathtools}
\usepackage{mathrsfs}
\usepackage{accents} 
\usepackage{hyperref}
\hypersetup{pdfstartview=}

\title{Spherical equidistribution in adelic lattices and applications
\thanks{The research leading to these results has received funding from the European Research Council under the European Union's Seventh Framework Programme (FP/2007-2013) / ERC Grant Agreement n. 291147.}}

\author{Daniel El-Baz}

\begin{document}

\maketitle

\newtheorem{defn}{Definition}[section]
\newtheorem{claim}{Claim}[section]
\providecommand*{\claimautorefname}{Claim}
\newtheorem{thm}{Theorem}[section]
\providecommand*{\thmautorefname}{Theorem}
\newtheorem{lemma}{Lemma}[section]
\providecommand*{\lemmaautorefname}{Lemma}
\newtheorem{cor}{Corollary}[section]
\providecommand*{\corautorefname}{Corollary}
\newtheorem{rmk}{Remark}[section]
\providecommand*{\rmkautorefname}{Remark}

\begin{abstract}
In this paper we study spherical equidistribution on the space of (translates of) adelic lattices, which we apply to understand the fine-scale statistics of the directions in the set of shifted primitive lattice points. We also apply our results to the distribution of the free path lengths in the Boltzmann--Grad limit for point sets such as (possibly non-rational) translates of the lattice points all of whose coordinates are squarefree.
Besides the equidistribution results for translates of expanding horospheres, a key ingredient is a probabilistic argument which allows us to tackle the technical difficulty of dealing with characteristic functions of compact sets with positive measure and empty interior.
\end{abstract}

\section{Introduction} \label{intro}
\subsection{Motivation} \label{motiv}

The cut-and-project method is a well-known tool to generate quasiperiodic point sets in $\mathbb{R}^d$, a notable example being the vertices of a Penrose tiling \cite{deBruijn1981, marklof_strombergsson_quasicrystals_fpl_2014}. For this and other typical examples \cite{BaakeGrimm}, the method consists in starting with a lattice in $\mathbb{R}^{d+m}$ for some integer $m \ge 1$ and {\em projecting} onto $\mathbb{R}^d$ those points of the lattice which ``make the {\em cut}'' as set by a so-called window, a compact subset of $\mathbb{R}^m$. We shall refer to point sets obtained in such a manner as Euclidean model sets, with the extra adjective ``regular'' when the window has non-empty interior, and ``weak'' otherwise. It also makes sense to allow for $\mathbb{R}^m$ to be replaced by an arbitrary locally compact abelian group \cite{Meyerbook, Moodyunifdist, BaakeGrimm}.

The purpose of this paper is to investigate certain questions in the situation of (possibly weak) ``adelic'' model sets, given by such a scheme in which our locally compact abelian group is $\mathbb{A}_f^d$, $d \ge 2$, where $\mathbb{A}_f$ denotes the finite adeles. A good example (of a weak adelic model set) to bear in mind is the set of primitive lattice points in $\mathbb{R}^d$, or indeed any --- possibly non-rational --- translate of this set. 

For simplicity, we describe two of our results precisely for the set \[ \mathcal{P} = \{ (x_1, \ldots, x_d) \in \mathbb{Z}^d \, : \, \gcd(x_1, \ldots, x_d) = 1 \} \subset \mathbb{Z}^d \] of visible lattice points, deferring the exact definition of adelic model sets along with the full statement of our theorems to \autoref{apps}.
This set $\mathcal{P}$ has density $\frac 1{\zeta(d)}$ (where $\zeta$ is the Riemann zeta function) and possesses the interesting property of containing arbitrarily large holes: given any $R > 0$, there exists a ball of radius $R$ not containing any visible lattice point --- a standard application of the Chinese remainder theorem.

Our first result deals with the distribution of free path lengths for the Lorentz gas in the Boltzmann--Grad limit, a first step in the study of the kinetic transport in this model.
There are a few classes of point sets for which this is understood. For random realisations of a Poisson point process, the distribution of the free path length is easily seen to be exponential, but the full kinetic transport is understood thanks to the work of Boldrighini, Bunimovich and Sinai \cite{boldrighini_bunimovich_sinai_1983}. In the case of the (Euclidean) lattice $\mathbb{Z}^d$, the distribution of the free path length is far less trivial  \cite{BourgainGolseWennberg1998,CagliotiGolse2003,BocaZaharescu2007, marklof_strombergsson_free_path_length_2010} and we also fully understand the kinetic transport through the work of Marklof and Str\"ombergsson \cite{marklof_strombergsson_boltzmann-grad_2011}.
Our inspiration for this paper is the case of regular cut-and-project sets in $\mathbb{R}^d$ as studied in \cite{marklof_strombergsson_quasicrystals_fpl_2014}, for which the internal space is another Euclidean space, say $\mathbb{R}^m$ for some positive integer $m$. We now describe the Lorentz gas model in our setting.

Place balls of radius $\rho>0$ centred at each point of $\mathcal{P}$ and denote by $\mathcal{S}_\rho$ the union of these balls. Consider the first time a particle starting from a point $\bm q \in \mathbb{R}^d$ and travelling along a straight line with initial direction $\bm v \in S^{d-1}$ hits one of those balls:
\begin{equation}\label{fpldef} \tau(\bm q, \bm v, \rho) = \inf\{t > 0 \, : \, \bm q + t\bm v \in \mathcal{S}_\rho \}. \end{equation}
A priori $\bm q$ may belong to $\mathcal{P}$, in which case we simply exclude the ball centred at $\bm q$ from $\mathcal{S}_\rho$ for the above definition to still be sensible.

In this setting, we are able to prove the existence of a limiting distribution as $\rho \to 0$ in the Boltzmann--Grad limit.

\begin{thm} \label{introfpl} For every $\bm q$ in $\mathbb{R}^d$, there exists a function $D_{\mathcal{P}, \bm q} \colon \mathbb{R}_{\ge 0} \to [0,1]$ such that for every Borel probability measure $\lambda$ on $S^{d-1}$ which is absolutely continuous with respect to the Lebesgue measure and every $\xi > 0$, 
\begin{equation} \lim_{\rho \to 0} \lambda(\{ \bm v \in S^{d-1} \, : \, \rho^{d-1} \tau(\bm q, \bm v, \rho) \ge \xi \}) = D_{\mathcal{P}, \bm q}(\xi).\end{equation} 
\end{thm}

We can actually derive an explicit formula for the limiting distribution in terms of random adelic lattices. This is stated and proved in \autoref{appfpl} for the general case of ``arithmetic cut-and-project sets''.

Following the exploration started by Baake, G\"otze, Huck and Jakobi \cite{baake_gotze_huck_jakobi_2014} for certain mathematical quasicrystals, one may also ask about the local statistics of directions in our point set $\mathcal{P}$, and this is the content of our second result.

To place ourselves in their context we further restrict to $d=2$ for now, but prove a general result, valid for $d \ge 2$, in \autoref{directions}.
Consider then a point $\bm \xi \in \mathbb{R}^2$ and look at the set of points of $\mathcal{P} + \bm \xi$ inside the open disc of radius $T > 0$ centred at the origin. Denote this set minus the origin by $\mathcal{P}_T$.
As $\bm x$ ranges through $\mathcal{P}_T$, we are interested in the distribution of $\frac {\bm x}{\| \bm x \|_2}$, counted with multiplicity.
For each $T$, this produces a sequence of $N = N(T)$ angles $\alpha_j = \alpha_j(T) \in \mathbb{T} = \mathbb{R}/\mathbb{Z}, \, j \in \{1, \ldots, N\}$.
Given $I \subset \mathbb{T}$ and $\alpha \in \mathbb{T}$ chosen uniformly at random, we look at the number of angles falling into a small interval randomly shifted by $\alpha$:
\begin{equation} \mathcal{N}_T(I, \alpha) = \# \left\{ j \le N \, : \, \alpha_j \in \frac 1N I + \alpha \right\}. \end{equation}

We may now state:
\begin{thm} \label{introdir}
For every $\bm \xi \in \mathbb{R}^2$, every $I \subset \mathbb{T}$ and every $\alpha$ distributed uniformly at random in $\mathbb{T}$, the random variable $\mathcal{N}_T(I, \alpha)$ has a limiting distribution as $T \to +\infty$. \end{thm}
This is in fact a special case of \autoref{dir}. As with the limiting distribution for the free path length, we also obtain an explicit formula for the limiting distribution in this case.

We note that by a general argument, presented in \cite[Section 2.1]{marklof_survey}, this implies the existence of a limiting gap distribution for the angles $(\alpha_j)$, meaning for every $x \ge 0$, the quantity 
\begin{equation} \label{defgapdist} \frac 1N \# \left\{ 1 \le j \le N \, : \, \alpha'_{j+1} - \alpha'_j \le \frac xN \right\} 
\end{equation}
converges as $N \to \infty$, where $(\alpha'_j)_{1 \le j \le N}$ is the non-decreasing reordering of $(\alpha_j)_{1 \le j \le N}$.

\begin{figure}[h]
	\centering
	{
		\includegraphics[width=80mm]{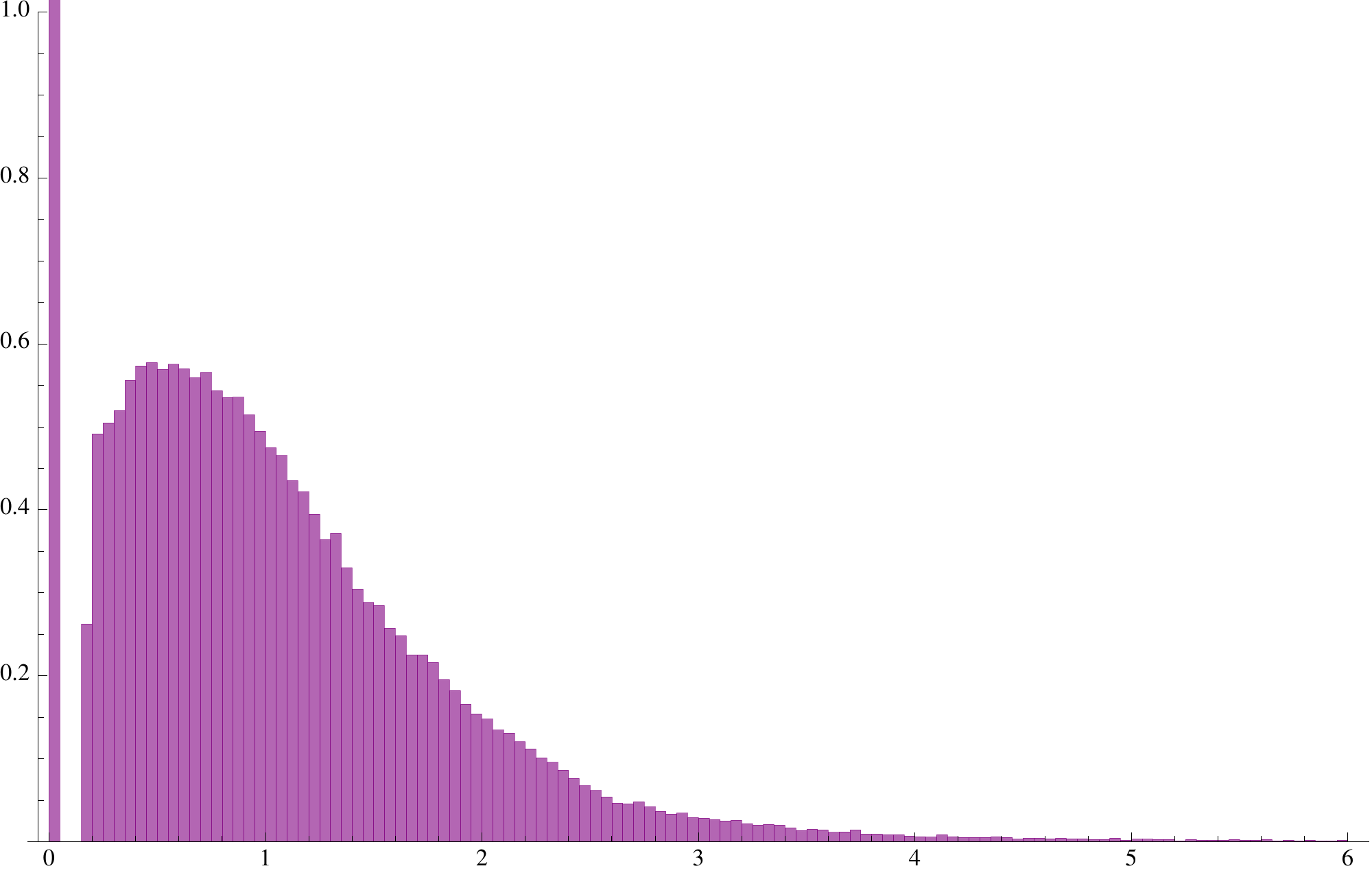}
	}
	\hspace{0mm}
	{
		\includegraphics[width=80mm]{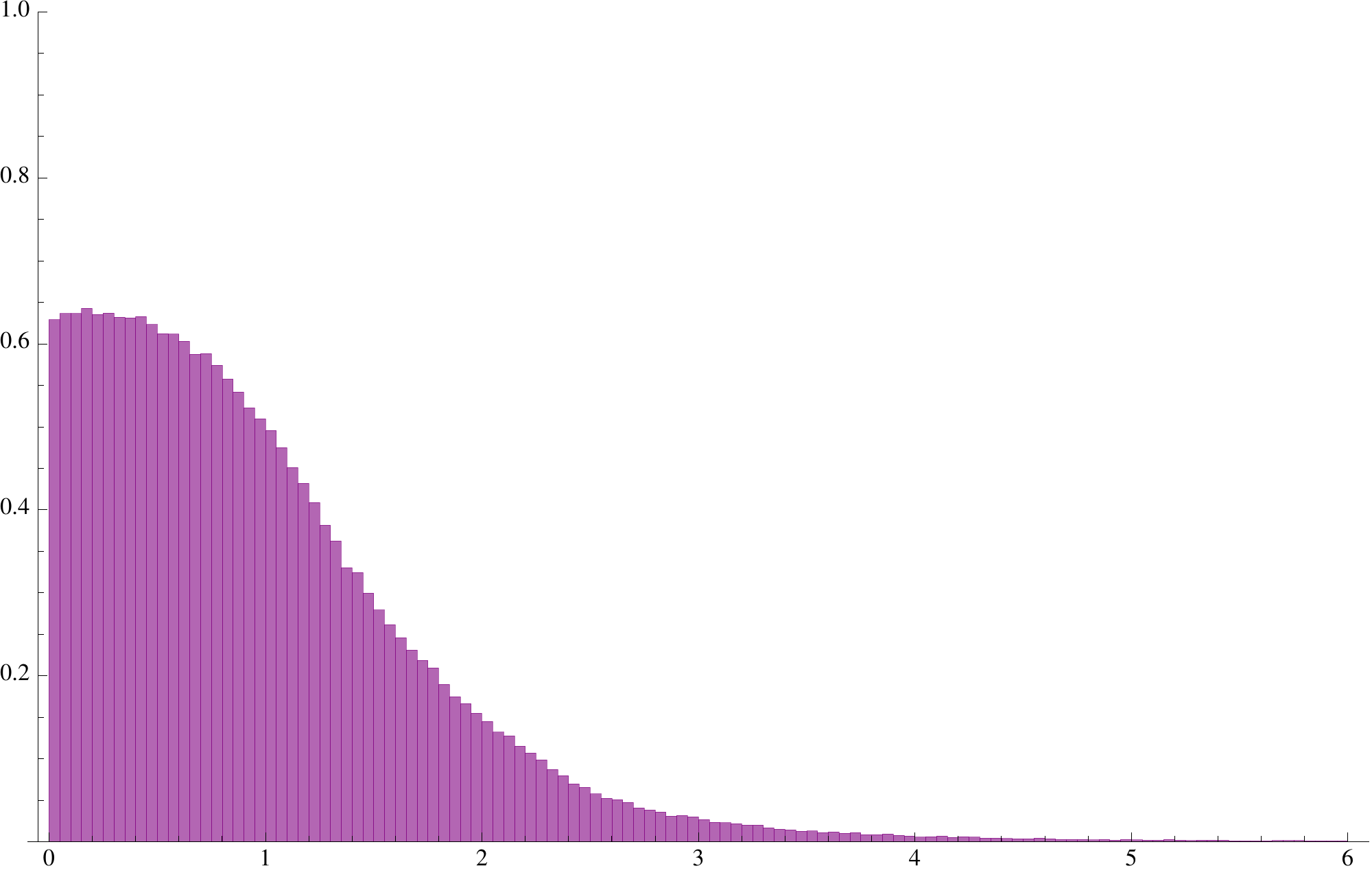}
	}
	\caption{The limiting gap distributions corresponding to the primitive lattice points shifted by $\left(\frac 12, \frac 12 \right)$ (top) and the primitive lattice points shifted by $(\sqrt{2}, \sqrt{3})$ (bottom), within a disc of radius $500$.}
\end{figure}

We should also point out that when $\bm \xi = \bm 0$, the local statistics of the primitive lattice points are intimately connected to the statistics of the classical Farey fractions, whose gap distribution was studied by Hall in 1970 \cite{Hall70}. The existence of a limiting gap distribution for the primitive lattice points was proved by Boca, Cobeli and Zaharescu in 2000 \cite{BCZ2000}.

\subsection{Main theorem}
We proceed to describe our main theorem, which asserts the existence of a limiting distribution for certain random point processes.
It implies \autoref{introfpl} and \autoref{introdir}, as shown in \autoref{apps}.

In what follows, $\delta$ denotes the diagonal embedding of $\mathbb{Q}$ into $\mathbb{A}$, 
\begin{align*} \delta \colon \mathbb{Q} &\hookrightarrow \mathbb{A} \\
x &\mapsto (x, x, x, \ldots), \end{align*}
while $\iota$ denotes the embedding of $\mathbb{R}$ into $\mathbb{A}$ through the first coordinate
\begin{align*} \iota \colon \mathbb{R} &\hookrightarrow \mathbb{A} \\
x &\mapsto (x, 0, 0, \ldots). \end{align*} 
Abusing notation, we shall also denote by $\iota$ the embedding of $\mathrm{SL}_d(\mathbb{R})$ into $\mathrm{SL}_d(\mathbb{A})$ via the first coordinate, \begin{align*} \iota \colon \mathrm{SL}_d(\mathbb{R}) &\hookrightarrow \mathrm{SL}_d(\mathbb{A}) \\
M &\mapsto (M, I_d, I_d, \ldots) \end{align*}
and by $\delta$ the diagonal embedding of $\mathrm{SL}_d(\mathbb{Q})$ into $\mathrm{SL}_d(\mathbb{A})$,
\begin{align*} \delta \colon \mathrm{SL}_d(\mathbb{Q}) &\hookrightarrow \mathrm{SL}_d(\mathbb{A}) \\
M &\mapsto (M, M, M, \ldots). \end{align*}

For a ring $R$, we define $\mathrm{ASL}_d(R) = \mathrm{SL}_d(R) \ltimes R^d$.
As above, we view $\mathrm{ASL}_d(\mathbb{R})$ as embedded into $\mathrm{ASL}_d(\mathbb{A})$ through the first coordinate and $\mathrm{ASL}_d(\mathbb{Q})$ diagonally, once again denoting those embeddings by $\iota$ and $\delta$.
The multiplication law on $\mathrm{ASL}$ is defined (thinking of vectors as row vectors) by $(M, \vecxi)(M', \vecxi') = (MM', \vecxi M' + \vecxi')$.

For $t >0$, we define $\Phi^t = \operatorname{diag}(e^{-(d-1)t}, e^t, \ldots, e^t)$.
We also fix a map $K \colon S^{d-1} \to \mathrm{SO}(d)$ satisfying \begin{equation} \forall \bm v \in S^{d-1}, \, \bm v K(\bm v) = (1, 0, \ldots, 0) \end{equation} and which is smooth on the sphere minus a point (an explicit example of such a map is given in \cite[Footnote 3, p. 1968]{marklof_strombergsson_free_path_length_2010}).

\begin{thm} \label{mainthm}
Let $\mathcal{A} \subset \mathbb{A}^d$ be a bounded Borel set with $m(\partial \mathcal{A}) = 0$. 
For every $\bm \alpha \in \mathbb{R}^d$, every $M \in \mathrm{SL}_d(\mathbb{A})$, every $r \in \mathbb{Z}_{\ge 0}$ and every Borel probability measure $\lambda$ on $S^{d-1}$ which is absolutely continuous with respect to the Lebesgue measure, 
\[
\lambda(\{ \bm v \in S^{d-1} \, : \, \#(\mathcal{A} \iota((\Phi^{-t}, \mathbf{0}) (K(\bm v), \mathbf{0}) (I_d, \bm \alpha)) \cap \delta(\mathbb{Q})^d M) \ge r \}) \] has a limit when $t$ tends to infinity and it is given by 
\[ \mu_{\mathrm{ASL}_d(\mathbb{Q}) \backslash \mathrm{ASL}_d(\mathbb{A}) }( \{ g \in \mathrm{ASL}_d(\mathbb{Q}) \backslash \mathrm{ASL}_d(\mathbb{A}) \, : \, \# (\mathcal{A} \cap \delta(\mathbb{Q})^d g) \ge r \} ) 
\] if $\bm \alpha \in \mathbb{R}^d \setminus \mathbb{Q}^d$
and by
\[
\mu_{\mathrm{SL}_d(\mathbb{Q}) \backslash \mathrm{SL}_d(\mathbb{A})}( \{ g \in \mathrm{SL}_d(\mathbb{Q}) \backslash \mathrm{SL}_d(\mathbb{A}) \, : \, \# (\mathcal{A} - (\bm 0, \bm \beta M_f) \cap \delta(\mathbb{Q})^d g \setminus \{ \mathbf{0} \}) \ge r \} ) \] if $\bm \alpha \in \mathbb{Q}^d$, where $\bm \beta = -\delta(\bm \alpha)_f$  \end{thm}

The space of adelic lattices can be viewed as a fibre bundle over the space of lattices, or in other words a space of ``marked'' lattices as coined by Marklof and Vinogradov in \cite{MarklofVinogradov2017}. In that setting they prove a spherical equidistribution result for every point in the base, but only almost every point in the fibre. However, the above \autoref{mainthm} holds for every point. \\

\textbf{Acknowledgement:} I would like to thank my PhD adviser, Prof Jens Marklof, for suggesting I investigate these questions as well as for his continued support and guidance.

\section{Preliminaries}
As mentioned in the case of the $p$-adic integers, a description as an inverse limit can sometimes be a convenient tool to reduce a problem to (possibly) many hopefully easier ones.
We give such a description here for the matrix rings discussed above.
In this section, we give a description of the relevant matrix rings for our purposes as an inverse limit. This description is a convenient reduction tool, as we demonstrate in \autoref{equid}.

\subsection{The special linear group}
For $N \in \mathbb{N}$, let $\Gamma(N) \subset \mathrm{SL}_d(\mathbb{Z})$ denote a principal congruence subgroup of $\mathrm{SL}_d(\mathbb{Z})$, i.e. $\Gamma(N) = \ker( \mathrm{SL}_d(\mathbb{Z}) \twoheadrightarrow \mathrm{SL}_d(\mathbb{Z}/N \mathbb{Z}))$.

The following statement is at the heart of the modern, adelic theory of automorphic representations, as presented in Gelbart's book \cite{Gelbart} for instance. It is not stated the way we do in that book, with the most closely-related result we were able to find in the classical literature being Proposition 3.3.1 in Bump's book \cite{Bump}.

\begin{lemma}\label{SLhomeo}
There is a homeomorphism 
\begin{equation} \mathrm{SL}_d(\mathbb{Q}) \backslash \mathrm{SL}_d(\mathbb{A}) \cong \varprojlim_{N \in \mathbb{N}} \Gamma(N) \backslash \mathrm{SL}_d(\mathbb{R}). \end{equation}
\end{lemma}

\begin{proof}
First, $(\Gamma(N) \backslash \mathrm{SL}_d(\mathbb{R}))_{N \in \mathbb{N}}$ is a projective system (of locally compact Hausdorff spaces): indeed, $M \mid N \implies \Gamma(N) \subset \Gamma(M)$.
Define, for $N \in \mathbb{N}$, the following compact open subgroups of $\mathrm{SL}_d(\mathbb{A}_f)$:
\begin{equation} K(N) = \ker(\mathrm{SL}_d(\widehat{\mathbb{Z}}) \to \mathrm{SL}_d(\mathbb{Z}/N \mathbb{Z})). \end{equation}

By the strong approximation theorem for $\mathrm{SL}_d$ --- whose  proof is elementary and already contained in Bourbaki's book on commutative algebra \cite[$\S 2, n^\circ 3$, Prop 4]{BourbakiAlgComm} ---  it follows that for every open subgroup $K$ of $\mathrm{SL}_d(\mathbb{A})$ and every place $p$ we have $\mathrm{SL}_d(\mathbb{Q})\mathrm{SL}_d(\mathbb{Q}_p)K = \mathrm{SL}_d(\mathbb{A})$.
In particular, we take $p = \infty$ and $K = \mathrm{SL}_d(\mathbb{R}) \times \prod_{p < \infty} K_p'$ where $K_p'$ is an open (finite index) subgroup of $K_p = \mathrm{SL}_d(\mathbb{Z}_p)$ and $K_p' = K_p$ for all but finitely many $p$.

More precisely, for $N \in \mathbb{N}$ with prime factorisation $N = \prod_{i = 1}^r p_i^{\alpha_i}$, take 
\begin{align*} \forall i \in \{1, \ldots, r\}, K_{p_i}' &= \ker(\mathrm{SL}_d(\mathbb{Z}_{p_i}) \to \mathrm{SL}_d(\mathbb{Z}_{p_i} / p_i^{\alpha_i} \mathbb{Z}_{p_i})), \\
 \text{for } p \ne p_i, K_p' &= K_p, \\
 K^N &= \mathrm{SL}_d(\mathbb{R}) \times \prod_{p < \infty} K_p' \\
\text{and } K(N) &= \prod_{p < \infty} K_p'. \end{align*}
We then have \begin{equation} \mathrm{SL}_d(\mathbb{A}) = \mathrm{SL}_d(\mathbb{Q}) \mathrm{SL}_d(\mathbb{R}) K(N). \end{equation}
Upon observing that $\mathrm{SL}_d(\mathbb{Q}) \cap K^N = \Gamma(N)$, this yields \begin{equation} \mathrm{SL}_d(\mathbb{Q}) \backslash \mathrm{SL}_d(\mathbb{A}) / K(N) = \Gamma(N) \backslash \mathrm{SL}_d(\mathbb{R}). \end{equation} \end{proof}

\begin{rmk} In the same way that $\mathrm{SL}_d(\mathbb{Z}) \backslash \mathrm{SL}_d(\mathbb{R})$ can be identified with the space of unimodular lattices in $\mathbb{R}^d$, $\mathrm{SL}_d(\mathbb{Q}) \backslash \mathrm{SL}_d(\mathbb{A})$ can be identified with the space of unimodular lattices in $\mathbb{A}^d$. We refer the reader to \cite[Section 0.1]{DeligneGL2} for a concise treatment of adelic lattices. \end{rmk}

\subsection{The special affine group} 
This is the corresponding statement for the affine special linear group.
\begin{lemma}\label{ASLhomeo} We have the following homeomorphism:
\begin{equation} \mathrm{ASL}_d(\mathbb{Q}) \backslash \mathrm{ASL}_d(\mathbb{A}) \cong \varprojlim_N (\Gamma(N) \ltimes N \mathbb{Z}^d) \backslash \mathrm{ASL}_d(\mathbb{R}). \end{equation}
 \end{lemma}

\begin{proof}
We combine \autoref{SLhomeo} with $\mathbb{Q} \backslash \mathbb{A} \cong \varprojlim_N N\mathbb{Z} \backslash \mathbb{R}$ and use the commutativity of inverse limits with products (see \cite[2.12]{Borceux}) to write:
\begin{align}
\mathrm{ASL}_d(\mathbb{Q}) \backslash \mathrm{ASL}_d(\mathbb{A}) &\cong \mathrm{SL}_d(\mathbb{Q}) \backslash \mathrm{SL}_d(\mathbb{A}) \times \mathbb{Q}^d \backslash \mathbb{A}^d \\
&\cong \varprojlim_{N} \Gamma(N) \backslash \mathrm{SL}_d(\mathbb{R}) \times \varprojlim_N (N \mathbb{Z} \backslash \mathbb{R})^d \\
&\cong \varprojlim_N (\Gamma(N) \backslash \mathrm{SL}_d(\mathbb{R}) \times (N\mathbb{Z}^d) \backslash \mathbb{R}^d) \\
&\cong \varprojlim_N (\Gamma(N) \ltimes N \mathbb{Z}^d) \backslash \mathrm{ASL}_d(\mathbb{R}).
\end{align}
\end{proof}

\subsection{Function spaces}

For later applications, it will be helpful not only to have a description for the matrix rings as above but one for certain spaces of continuous functions over them.
In what follows, for $N \ge 1$, $(X,X_N) = ( \mathrm{SL}_d(\mathbb{Q}) \backslash \mathrm{SL}_d(\mathbb{A}),  \Gamma(N) \backslash \mathrm{SL}_d(\mathbb{R}))$.

For a topological space $T$, we denote by $\mathcal{C}_0(T)$ the set of real-valued continuous functions on $T$ vanishing at infinity.
 
\begin{lemma}\label{SLindlim} Any $f \in \mathcal{C}_0(X)$ can be approximated uniformly by functions $f_N \in \mathcal{C}_0(X_N)$.
\end{lemma}

\begin{proof}
This is a consequence of \autoref{SLhomeo} and the fact that the canonical projection maps $\pi_N \colon X \to X_N$ are proper.

Properness ensures that if we pull back a function vanishing at infinity on $X_N$ up to $X$, it vanishes at infinity there as well.
The inverse limit in \autoref{SLhomeo} is in the category of locally compact Hausdorff spaces. With proper maps, that is dual to the category of commutative $C^*$-algebras. So $\mathcal{C}_0(X)$ is the direct limit of the $\mathcal{C}_0(X_N)$, and then it is immediate that functions in $\mathcal{C}_0(X)$ can be approximated by functions in $\mathcal{C}_0(X_N)$.

More precisely, since the maps $\pi_N$ are proper we have a duality (meaning a contravariant equivalence) between our category of spaces and that of commutative $C^*$-algebras. Because of the contravariance, limits in the former correspond to colimits in the latter. So by the assertion that $X = \varprojlim_N X_N$, we have $\mathcal{C}_0(X) = \varinjlim_N \mathcal{C}_0(X_N)$. Now the approximation statement is immediate: (one can either look at \cite[II.8.2]{Blackadar} or) it is easy to convince oneself by considering the closure of the unions of images of the $\mathcal{C}_0(X_N)$ in $\mathcal{C}_0(X)$ and proving that it has the universal property of a colimit, so it must coincide with $\mathcal{C}_0(X)$. In other words, because we have that inductive limit, we have (*-homomorphisms) $i_N \colon \mathcal{C}_0(X_N) \to \mathcal{C}_0(X)$ and the union of the $i_N(\mathcal{C}_0(X_N))$ is dense in $\mathcal{C}_0(X)$.

\end{proof}

The proof of the corresponding statement for the special affine group is entirely analogous.

\section{A Siegel--Weil formula} \label{adelicSiegel}

The final result presented in this section can be seen as a striking extension of the analogy between $\mathbb{R}/\mathbb{Z}$ and $\mathbb{A}/\mathbb{Q}$.
We first recall Siegel's celebrated mean value theorem \cite{SiegelMVT45} on the space of lattices. 

\begin{thm}\label{Siegelformula} If $f \in L^1(\mathbb{R}^d)$, then 
\begin{equation} \int_{\mathrm{SL}_d(\mathbb{Z}) \backslash \mathrm{SL}_d(\mathbb{R})} \sum_{\vecm \in \mathbb{Z}^d \setminus \{\mathbf 0\}} f(\vecm M) d \mu_{\mathrm{SL}_d(\mathbb{Z}) \backslash \mathrm{SL}_d(\mathbb{R})}(M) = \int_{\mathbb{R}^d} f(\vecx) d\vecx. \end{equation} \end{thm}

In the above formula, the measure $\mu_{\mathrm{SL}_d(\mathbb{Z}) \backslash \mathrm{SL}_d(\mathbb{R})}$ is normalised to be a probability measure. In the proof below, we shall instead use the Haar measure coming from the Iwasawa decomposition, which will allow us --- following Siegel --- to compute the normalisation constant, that is, the volume of $\mathrm{SL}_d(\mathbb{Z}) \backslash \mathrm{SL}_d(\mathbb{R})$: $\prod_{k=2}^d \zeta(k)$, where $\zeta$ is the Riemann zeta function.
In addition, that formula was used by Siegel as the basis for a probabilistic method argument to solve a conjecture of Minkowski in the geometry of numbers.

In 1946, Weil viewed Siegel's result in the context of his integration theory on topological groups \cite{Weil1946} and gave a proof which essentially yields the adelic version as well. 
We sketch Weil's proof for $d=2$, which already contains the crucial ideas and helps us motivate the adelic result.

\begin{proof} 
Assume $f \in \mathcal{C}_c(\mathbb{R}^2)$ is non-negative. Let $G = \mathrm{SL}_2(\mathbb{R})$ and $\Gamma = \mathrm{SL}_2(\mathbb{Z})$.
We define, for $M \in G$, \begin{equation} F(M) = \sum_{\vecm \in \mathbb{Z}^2} f(\vecm M). \end{equation} 
Note that, by definition, $F$ is left-$\Gamma$-invariant.
The fact that for any primitive vector there exists a matrix in $\Gamma$ whose last column is that vector along with the fact that the stabiliser of  $(0,1)$ in $\Gamma$ is $N(\mathbb{Z}) = \left\{ \begin{pmatrix} 1 & x \\ 0 & 1 \end{pmatrix} \, : \, x \in \mathbb{Z} \right\}$ implies that $\mathbb{Z}^2  \setminus \{ \bm 0 \}$ is in bijection with $\mathbb{N} \times N(\mathbb{Z}) \backslash \Gamma$.
We can thus split the integral as follows: 
\begin{align} \int_{\Gamma \backslash G} F d\mu =& \int_{\Gamma \backslash G} f(\bm 0) d\mu + \int_{\Gamma \backslash G} \sum_{\bm m \in \mathbb{Z}^2 \setminus \{ \bm 0 \}} f(\bm m M) d\mu(M) \\
=&  \mu(\Gamma \backslash G) f(\bm 0) + \sum_{l \in \mathbb{N}} \int_{N(\mathbb{Z})\backslash G} f(l \cdot (0, 1) M) d\mu. \end{align}
If we define $S = \sum_{l \in \mathbb{N}} \int_{N(\mathbb{Z})\backslash G} f(l \cdot (0, 1) M) d\mu$ and $N = \left\{ \begin{pmatrix} 1 & x \\ 0 & 1 \end{pmatrix} \, : \, x \in \mathbb{R} \right\}$, then we can use the Iwasawa decomposition to compute $S$: 
\begin{equation} S = \mu(N(\mathbb{Z}) \backslash N) \zeta(2) \int_{\mathbb{R}^2} f(\bm x) d \bm x. \end{equation}
However $N(\mathbb{Z}) \backslash N$ can be identified with $\mathbb{Z} \backslash \mathbb{R}$ and so has measure $1$.
Finally we obtain: 
\begin{equation} \int_{\Gamma \backslash G} \sum_{\vecm \in \mathbb{Z}^2} f(\vecm M) d\mu(M) = \mu(\Gamma \backslash G) f(\bm 0) +  \zeta(2) \hat{f}(\bm 0)\end{equation}
where $\hat{f}$ is the Fourier transform of $f$.
Now Weil's trick is to apply Poisson summation, which reads:
\[ \sum_{\bm m \in \mathbb{Z}^2} f(\vecm M) = \sum_{\bm n \in \mathbb{Z}^2} \hat{f}(\bm n \; ^tM^{-1}). \]
Since $M \mapsto \; ^tM^{-1}$ is an automorphism of $G$ which maps $\Gamma$ to itself, we can reproduce the above calculation with $\hat{f}$ instead of $f$, and conclude, noting that $\hat{\hat{f}}(\bm 0)=f(\bm 0)$ by Fourier inversion: 
\[ \mu(\Gamma \backslash G) f(\bm 0) +  \zeta(2) \hat{f}(\bm 0) = \mu(\Gamma \backslash G) \hat{f}(\bm 0) +  \zeta(2) f(\bm 0). \]
Choosing $f$ such that $f(\bm 0) \ne \hat{f}(\bm 0)$ allows us to deduce that $\mu(\Gamma \backslash G) = \zeta(2)$.

\end{proof}

As Weil himself writes in the \emph{Commentaire} of his \emph{\OE{}uvres Scientifiques} \cite{weil1979oeuvres} pertaining to that 1946 paper:

``Je constatai que par l'application de quelques r\'{e}sultats g\'{e}n\'{e}raux, et par l'usage de la sommation de Poisson, on pouvait simplifier la d\'{e}monstration de Siegel [\ldots]. Il n'y a eu aucune difficult\'{e}, par la suite, \`{a} transposer au cas `ad\'{e}lique' la m\'{e}thode que j'y employais et \`{a} l'appliquer au calcul du `nombre de Tamagawa'.''

Indeed, the same type of calculation as in the above proof can be performed adelically. 
We note one key difference: $\mathrm{SL}_d(\mathbb{A})$ does not act transitively on $\mathbb{A}^d \setminus \{ \bm 0 \}$ ($\mathbb{A}$ is merely a ring, not a field).
However, $\mathrm{SL}_d(\mathbb{A})$ does act transitively on the smaller, but nonetheless Zariski-open, set $X=U(\mathbb{A})$, where $U$ is the complement of the origin in affine $\mathbb{Q}$-space. Still, the complement has measure $0$ and so Weil can make the argument work \cite[Chapter III, 3.4]{WeilBook}. In the adelic case, the $\zeta$ factor in the above calculation is always $1$, which is usually phrased as the assertion that the Tamagawa number of $\mathrm{SL}_d$ over $\mathbb{Q}$ is $1$. For more on Tamagawa numbers (including their proper definition), we refer to Clozel's S\'eminaire Bourbaki \cite{ClozelTamagawa}.

We state the Siegel--Weil formula that ensues:
\begin{thm}\label{SiegelWeilformula} If $f \in L^1(\mathbb{A}^d)$, then 
\begin{equation} \int_{\mathrm{SL}_d(\mathbb{Q}) \backslash \mathrm{SL}_d(\mathbb{A})} \sum_{\vecq \in \mathbb{Q}^d \setminus \{\mathbf 0\}} f(\vecq M) d \mu_{\mathrm{SL}_d(\mathbb{Q}) \backslash \mathrm{SL}_d(\mathbb{A})}(M) = \int_{\mathbb{A}^d} f(\vecx) d\vecx. \end{equation} \end{thm}

\section{Equidistribution theorems} \label{equid}
\subsection{On the space of adelic lattices}

For $\bm x \in \mathbb{R}^{d-1}$, define $n(\bm x) = \begin{pmatrix} 1 & \bm x \\ ^t \bm 0 & I_{d-1} \end{pmatrix} \in \mathrm{SL}_d(\mathbb{R})$ and recall that, for $t>0$, we defined the diagonal flow $\Phi^t = \operatorname{diag}(e^{-(d-1)t}, e^t, \ldots, e^t)$.


\begin{thm}\label{sladelicequid} For every bounded real-valued continuous function $f$ on $\mathrm{SL}_d(\mathbb{Q})\backslash\mathrm{SL}_d(\mathbb{A})$, every $M \in \mathrm{SL}_d(\mathbb{A})$ and every Borel probability measure $\lambda$ on $\mathbb{R}^{d-1}$ which is absolutely continuous with respect to the Lebesgue measure on $\mathbb{R}^{d-1}$,
\begin{equation} \lim_{t \to \infty} \int_{\mathbb{R}^{d-1}} f(M \iota(n(\bm x) \Phi^t)) d \lambda(\bm x) = \int_{\mathrm{SL}_d(\mathbb{Q}) \backslash \mathrm{SL}_d(\mathbb{A})} f d \mu_{\mathrm{SL}_d(\mathbb{Q}) \backslash \mathrm{SL}_d(\mathbb{A})}. \end{equation} \end{thm}

\begin{proof}
Let $X = \mathrm{SL}_d(\mathbb{Q}) \backslash \mathrm{SL}_d(\mathbb{A})$ and, for $N \in \mathbb{N} \, , \,X_N = \Gamma(N) \backslash \mathrm{SL}_d(\mathbb{R})$. Denote by $\mu$ and $\mu_N$ the corresponding Haar measures.
Let $f \in \mathcal{C}_0(X)$ where $\mathcal{C}_0$ denotes the set of real-valued continuous functions vanishing at infinity. Let also $\lambda \ll \operatorname{Leb}_{d-1}$ be a Borel probability measure on $\mathbb{R}^{d-1}$.
Recall that we have (\autoref{SLhomeo}) \begin{equation}\label{homeo} X \cong \varprojlim_{N \in \mathbb{N}} X_N. \end{equation} 
We know --- by using the mixing property of the diagonal flow $\Phi^t$ \cite{Mooreergodicity}, as done in \cite{EskinMcMullen}, for instance --- that for every $N \in \mathbb{N}$, every $M_\infty \in \mathrm{SL}_d(\mathbb{R})$ and every $g \in \mathcal{C}_0(X_N)$,
\begin{equation}\label{mixing} \lim_{t \to \infty} \int_{\mathbb{R}^{d-1}} g(M_\infty n(\bm x) \Phi^t) d \lambda(\bm x) = \int_{X_N} g d \mu_N.\end{equation}
The result therefore follows right away when $f$ factors through $\pi_N \colon X \to X_N$ (so that it is actually a function on $X_N$). 
Indeed, the expression we want then simply reduces to the one we know: $\int_X f d\mu = \int_{X_N} f d\mu_N$.
So now the idea is to approximate every $f$ by such functions $f$ that ``come from'' $X_N$ (for ever larger $N$).
Let $\varepsilon > 0$.
By the claim about inverse limits \eqref{homeo} (more precisely, \autoref{SLindlim}), we can find $N \in \mathbb{N}$ and $f_N \in \mathcal{C}_0(X_N)$ such that \begin{equation}\label{approx} \forall M_0 \in X, \, |f(M_0) - f_N(\pi_N(M_0))| \le \frac {\varepsilon}3. \end{equation}
In particular, $\forall t > 0, \forall \bm x \in \mathbb{R}^{d-1}, \, |f(\iota(M_t(\bm x))) - f_N(\pi_N(\iota(M_t(\bm x))))| \le \frac {\varepsilon}3$, 
where we define $M_t(\bm x) = M_\infty n(\bm x) \Phi^t$.
Since $\lambda$ is a probability measure, this yields
\begin{equation} \forall t > 0, \, \left| \int_{\mathbb{R}^{d-1}} f(\iota(M_t(\bm x))) - f_N(\pi_N(\iota(M_t(\bm x)))) d\lambda(\bm x) \right| \le \frac {\varepsilon}3. \end{equation}
By \eqref{mixing}, we can find $t_0 > 0$ such that for all $t \ge t_0$, \begin{equation} \left| \int_{\mathbb{R}^{d-1}} f_N(\pi_N(\iota(M_t(\bm x)))) d \lambda(\bm x) - \int_{X_N} f_N d \mu_N \right| \le \frac {\varepsilon}3. \end{equation}
Now we get, for $t \ge t_0$,
\begin{align*} & \left| \int_{\mathbb{R}^{d-1}} f(\iota(M_t(\bm x))) d\lambda(\bm x) - \int_X f d\mu \right| \\ &\le \left|\int_{\mathbb{R}^{d-1}} f(\iota(M_t(\bm x))) d\lambda(\bm x) - \int_{\mathbb{R}^{d-1}} f_N(\pi_N(\iota(M_t(\bm x)))) d\lambda(\bm x) \right| \\ &\phantom{sp}+ \left|\int_{\mathbb{R}^{d-1}} f_N(\pi_N(\iota(M_t(\bm x)))) d\lambda(\bm x) - \int_{X_N} f_N d\mu_N \right| \\ &\phantom{sp}+ \left| \int_{X_N} f_N d\mu_N - \int_X f d\mu \right| \\ &\le \frac {\varepsilon}3 + \frac {\varepsilon}3 + \frac {\varepsilon}3 = \varepsilon. \end{align*}
The last $\frac {\varepsilon}3$ comes from:
\begin{equation} \left| \int_{X_N} f_N d\mu_N - \int_X f d\mu \right| =  \left| \int_X f_N \circ \pi_N d\mu - \int_X f d\mu \right| \le \int_X |f_N  \circ \pi_N - f| d\mu \le \frac {\varepsilon}3. \end{equation} We used the fact that $\mu_N = (\pi_N)_* \mu$ to get the first equality, while the last inequality follows from \eqref{approx} and the fact that $\mu$ is a probability measure. 

This concludes the proof when $f \in \mathcal{C}_0(X)$.
It is now a standard approximation argument to pass from the space of continuous functions vanishing at infinity $\mathcal{C}_0$ to the space of bounded continuous functions $\mathcal{C}_b$ given that the measure spaces involved are all probability spaces. We sketch that argument here for the sake of completeness. Because the measures involved are probability measures, the theorem immediately follows for constant functions $f$. This allows us to extend the result to continuous functions which are constant outside some compact set.
Now if $f$ is in $\mathcal{C}_b(X)$ and $\varepsilon > 0$, we find continuous functions on $X$, $f^-$ and $f^+$, which are constant outside some compact set and satisfy: $f^- \le f \le f^+$ and $\| f^+ - f^- \|_{L^1(X, \mu)} \le \varepsilon.$ Applying the result to those functions gives it for $f$ as well.

\end{proof}

To prove \autoref{slrandhoro}, it is helpful to state the following consequence of the above theorem explicitly:
\begin{cor}\label{limsupliminf} For every $M \in \mathrm{SL}_d(\mathbb{A})$, every Borel probability measure $\lambda$ on $\mathbb{R}^{d-1}$ which is absolutely continuous with respect to the Lebesgue measure on $\mathbb{R}^{d-1}$ and every measurable subset $\mathcal{E} \subset \mathrm{SL}_d(\mathbb{Q}) \backslash \mathrm{SL}_d(\mathbb{A})$, we have 
\begin{equation} \limsup_{t \to \infty} \lambda( \{ \bm x \in \mathbb{R}^{d-1} \, : \, M \iota(n(\bm x) \Phi^t) \in \mathcal{E} \}) \le \mu_{\mathrm{SL}_d(\mathbb{Q}) \backslash \mathrm{SL}_d(\mathbb{A})}(\overline{\mathcal{E}}) \end{equation}
and 
\begin{equation} \liminf_{t \to \infty} \lambda( \{ \bm x \in \mathbb{R}^{d-1} \, : \, M \iota(n(\bm x) \Phi^t) \in \mathcal{E} \}) \ge \mu_{\mathrm{SL}_d(\mathbb{Q}) \backslash \mathrm{SL}_d(\mathbb{A})}(\accentset{\circ}{\mathcal{E}}). \end{equation} \end{cor}

\begin{proof}
We write
\begin{equation}  \lambda( \{ \bm x \in \mathbb{R}^{d-1} \, : \, M \iota(n(\bm x) \Phi^t) \in \mathcal{E} \}) = \int_{\mathbb{R}^{d-1}} \chi_\mathcal{E}(M \iota(n(\bm x) \Phi^t)) d\lambda(\bm x). \end{equation}
Since $\mathcal{E} \subset \overline{\mathcal{E}}$, we have 
\begin{equation} \lambda( \{ \bm x \in \mathbb{R}^{d-1} \, : \, M \iota(n(\bm x) \Phi^t) \in \mathcal{E} \}) \le \int_{\mathbb{R}^{d-1}} \chi_{\overline{\mathcal{E}}}(M \iota(n(\bm x) \Phi^t)) d\lambda(\bm x). \end{equation} 
It now follows from \autoref{sladelicequid} in conjunction with the portmanteau theorem ($\overline{\mathcal{E}}$ is closed) that \begin{equation} \limsup_{t \to \infty} \lambda( \{ \bm x \in \mathbb{R}^{d-1} \, : \, M \iota(n(\bm x) \Phi^t) \in \mathcal{E} \}) \le \mu_{\mathrm{SL}_d(\mathbb{Q}) \backslash \mathrm{SL}_d(\mathbb{A})}(\overline{\mathcal{E}}). \end{equation}
The lower bound is obtained in a similar fashion.
\end{proof}
In fact, to obtain \autoref{mainthm} as announced in the introduction, the relevant equidistribution theorem is the following spherical equidistribution result:
\begin{thm} \label{sladelicequidrot} For every bounded continuous function $f$ on $\mathrm{SL}_d(\mathbb{Q})\backslash\mathrm{SL}_d(\mathbb{A})$, every $M \in \mathrm{SL}_d(\mathbb{A})$ and every Borel probability measure $\lambda$ on $S^{d-1}$ which is absolutely continuous with respect to the Lebesgue measure,
\[ \lim_{t \to \infty} \int_{S^{d-1}} f(M \iota(K(\bm v) \Phi^t)) d \lambda(\bm v) = \int_{\mathrm{SL}_d(\mathbb{Q}) \backslash \mathrm{SL}_d(\mathbb{A})} f d \mu_{\mathrm{SL}_d(\mathbb{Q}) \backslash \mathrm{SL}_d(\mathbb{A})}. \] \end{thm}
\begin{proof}
The proof follows the same lines as that of \autoref{sladelicequid}, relying on the fact (see \cite[Section 5.4]{marklof_strombergsson_free_path_length_2010})
that for every Borel probability measure $\lambda$ on $S^{d-1}$ which is absolutely continuous with respect to the Lebesgue measure, every $N \in \mathbb{N}$, every bounded continuous  $f \colon \mathrm{SL}_d(\mathbb{Z}) \backslash \mathrm{SL}_d(\mathbb{R}) \to \mathbb{R}$ and every $M_\infty \in \mathrm{SL}_d(\mathbb{R})$,
\[ \lim_{t \to \infty} \int_{S^{d-1}} f(M K(\bm v) \Phi^t) d\lambda(\bm v) = \int_{\mathrm{SL}_d(\mathbb{Z}) \backslash \mathrm{SL}_d(\mathbb{R})} f d\mu_{\mathrm{SL}_d(\mathbb{Z}) \backslash \mathrm{SL}_d(\mathbb{R})}. \]
instead of \eqref{mixing}. We define $M_t(\bm v) = M_\infty K(\bm v) \Phi^t$ this time and go through with the argument.
\end{proof}

We also state its corollary, which can be proved using \autoref{sladelicequidrot} in the same way as \autoref{limsupliminf} was proved using \autoref{sladelicequid}.

\begin{cor}\label{limsupliminfrot} For every $M \in \mathrm{SL}_d(\mathbb{A})$, every Borel probability measure $\lambda$ on $S^{d-1}$ which is absolutely continuous with respect to the Lebesgue measure on $S^{d-1}$ and every measurable subset $\mathcal{E} \subset \mathrm{SL}_d(\mathbb{Q}) \backslash \mathrm{SL}_d(\mathbb{A})$, we have 
\begin{equation} \limsup_{t \to \infty} \lambda( \{ \bm v \in S^{d-1} \, : \, M \iota(K(\bm v) \Phi^t) \in \mathcal{E} \}) \le \mu_{\mathrm{SL}_d(\mathbb{Q}) \backslash \mathrm{SL}_d(\mathbb{A})}(\overline{\mathcal{E}}) \end{equation}
and 
\begin{equation} \liminf_{t \to \infty} \lambda( \{ \bm v \in S^{d-1} \, : \, M \iota(K(\bm v) \Phi^t) \in \mathcal{E} \}) \ge \mu_{\mathrm{SL}_d(\mathbb{Q}) \backslash \mathrm{SL}_d(\mathbb{A})}(\accentset{\circ}{\mathcal{E}}). \end{equation} \end{cor}

\subsection{On the space of affine adelic lattices}

We now state the equidistribution theorems for affine adelic lattices. These follow from the ones in \cite{marklof_strombergsson_free_path_length_2010} in the case of $\bm \alpha \in \mathbb{R}^d \setminus \mathbb{Q}^d$, thanks to the reduction strategy developed in the proof of \autoref{sladelicequid} and \autoref{ASLhomeo}.

\begin{thm}\label{asladelicequid} For every bounded continuous function $f$ on $\mathrm{ASL}_d(\mathbb{Q})\backslash\mathrm{ASL}_d(\mathbb{A})$, every $M \in \mathrm{SL}_d(\mathbb{A})$, every Borel probability measure $\lambda$ on $\mathbb{R}^{d-1}$ which is absolutely continuous with respect to the Lebesgue measure on $\mathbb{R}^{d-1}$ and every $\bm \alpha \in \mathbb{A}^d$,
\[  \nu_{\bm \alpha}(f) = \lim_{t \to \infty} \int_{\mathbb{R}^{d-1}} f((I_d, \bm \alpha) (M, \bm 0) \iota((n(\bm x) \Phi^t, \bm 0))) d \lambda(\bm x) \]
exists and is given by
\[ \nu_{\bm \alpha}(f) = \begin{cases}  \int_{\mathrm{ASL}_d(\mathbb{Q}) \backslash \mathrm{ASL}_d(\mathbb{A})} f d \mu_{\mathrm{ASL}_d(\mathbb{Q}) \backslash \mathrm{ASL}_d(\mathbb{A})} &\mbox{if } \bm \alpha_\infty \in \mathbb{R}^d \setminus \mathbb{Q}^d \\
\int_{\mathrm{SL}_d(\mathbb{Q}) \backslash \mathrm{SL}_d(\mathbb{A})} f(g (I_d, (\bm 0, \bm \beta M_f))) d\mu_{\mathrm{SL}_d(\mathbb{Q}) \backslash \mathrm{SL}_d(\mathbb{A})}(g) &\mbox{if } \bm \alpha_\infty \in \mathbb{Q}^d \end{cases}, \] 
where $\bm \beta = \bm \alpha_f - \delta(\bm \alpha_\infty)_f$. \end{thm}


\begin{proof}
We begin with the case of $\bm \alpha_\infty \in \mathbb{R}^d \setminus \mathbb{Q}^d$.

We define for $N \in \mathbb{N}$, $\widetilde{X_N} = (\Gamma(N) \ltimes N \mathbb{Z}^d) \backslash \mathrm{ASL}_d(\mathbb{R})$ (where $\Gamma(N)$ is a principal congruence subgroup in $\mathrm{SL}_d(\mathbb{Z})$) and denote by $\widetilde{\mu_N}$ the corresponding Haar measure.
Following the proof of \autoref{sladelicequid} and using \autoref{ASLhomeo} instead of \autoref{SLhomeo}, the claim boils down to this statement: for every $M_\infty \in \mathrm{SL}_d(\mathbb{R})$, every $\bm \alpha_\infty \in \mathbb{R}^d \setminus \mathbb{Q}^d$, every $N \in \mathbb{N}$ and every $g \in C_b(\widetilde{X_N})$,
\[ \lim_{t \to \infty} \int_{\mathbb{R}^{d-1}} g((I_d, \bm \alpha_\infty)(M_\infty, \bm 0)(n(\bm x) \Phi^t, \bm 0)) d\lambda(\bm x) = \int_{\widetilde{X_N}} g d\widetilde{\mu_N}. \]
This is however the content of \cite[Theorem 5.2]{marklof_strombergsson_free_path_length_2010}.

If $\bm \alpha_\infty \in \mathbb{Q}^d$, then for $\Gamma = \mathrm{ASL}_d(\mathbb{Q})$ we have $(I_d, \delta(\bm \alpha_\infty)) \in \Gamma$ and thus
\begin{align}  &\Gamma(I_d, \bm \alpha) (M, \bm 0) \iota((n(\bm x) \Phi^t, \bm 0)) \\
&= \Gamma(I_d, -\delta(\bm \alpha_\infty))(I_d, \bm \alpha) M \iota(n(\bm x) \Phi^t)  \\
&=  \Gamma \iota(M_\infty n(\bm x) \Phi^t) (I_d, (\bm 0, \bm \beta)) (I_d, M_f)  \end{align}
where $\bm \beta$ is defined as above. We can therefore conclude that this case reduces to \autoref{sladelicequid} upon observing that \[ (I_d, (\bm 0, \bm \beta)) (I_d, M_f) = (I_d, M_f)(I_d, \bm \beta M_f). \]

\end{proof}

We once again state the following consequence as a separate corollary.

\begin{cor}
\label{limsupliminfasl} For every $M \in \mathrm{SL}_d(\mathbb{A})$, every Borel probability measure $\lambda$ on $\mathbb{R}^{d-1}$ which is absolutely continuous with respect to the Lebesgue measure on $\mathbb{R}^{d-1}$, every measurable subset $\mathcal{E} \subset \mathrm{ASL}_d(\mathbb{Q}) \backslash \mathrm{ASL}_d(\mathbb{A})$ and every $\bm \alpha \in \mathbb{A}^d$, we have 
\begin{align*} & \limsup_{t \to \infty} \lambda( \{ \bm x \in \mathbb{R}^{d-1} \, : \, (I_d, \bm \alpha) (M, \bm 0) \iota((n(\bm x) \Phi^t, \bm 0)) \in \mathcal{E} \}) \\ &\le  \begin{cases} \mu_{\mathrm{ASL}_d(\mathbb{Q}) \backslash \mathrm{ASL}_d(\mathbb{A})}(\overline{\mathcal{E}}) &\mbox{if } \bm \alpha_\infty \in \mathbb{R}^d \setminus \mathbb{Q}^d \\
\mu_{\mathrm{SL}_d(\mathbb{Q}) \backslash \mathrm{SL}_d(\mathbb{A})}(\{ g \in \mathrm{SL}_d(\mathbb{Q}) \backslash \mathrm{SL}_d(\mathbb{A}) \, : \, g (I_d, (\bm 0, \bm \beta M_f)) \in \overline{\mathcal{E}} \})  &\mbox{if } \bm \alpha_\infty \in \mathbb{Q}^d \end{cases} 
\end{align*}
and 
\begin{align*} &\liminf_{t \to \infty} \lambda( \{ \bm x \in \mathbb{R}^{d-1} \, : \, (I_d, \bm \alpha) (M, \bm 0) \iota((n(\bm x) \Phi^t, \bm 0)) \in \mathcal{E} \}) \\
&\ge \begin{cases} \mu_{\mathrm{ASL}_d(\mathbb{Q}) \backslash \mathrm{ASL}_d(\mathbb{A})}(\accentset{\circ}{\mathcal{E}}) &\mbox{if } \bm \alpha_\infty \in \mathbb{R}^d \setminus \mathbb{Q}^d \\
\mu_{\mathrm{SL}_d(\mathbb{Q}) \backslash \mathrm{SL}_d(\mathbb{A})}(\{ g \in \mathrm{SL}_d(\mathbb{Q}) \backslash \mathrm{SL}_d(\mathbb{A}) \, : \, g (I_d, (\bm 0, \bm \beta M_f)) \in \accentset{\circ}{\mathcal{E}} \})  &\mbox{if } \bm \alpha_\infty \in \mathbb{Q}^d
, \end{cases} 
\end{align*}
where $\bm \beta = \bm \alpha_f - \delta(\bm \alpha_\infty)_f$.
\end{cor}

As in the previous subsection, we have analogues for spherical equidistribution.
\begin{thm} \label{asladelicsphequid} For every bounded continuous function $f$ on $\mathrm{ASL}_d(\mathbb{Q})\backslash\mathrm{ASL}_d(\mathbb{A})$, every $M \in \mathrm{SL}_d(\mathbb{A})$, every Borel probability measure $\lambda$ on $S^{d-1}$ which is absolutely continuous with respect to the Lebesgue measure and every $\bm \alpha \in \mathbb{A}^d$,
\[ \lim_{t \to \infty} \int_{S^{d-1}} f((I_d, \bm \alpha) (M, \bm 0) \iota((K(\bm v) \Phi^t, 0)))) d \lambda(\bm v) \] exists and is also given by $\nu_{\bm \alpha}(f)$ as defined in \autoref{asladelicequid}. \end{thm}

\begin{proof}
We again define, for $N \in \mathbb{N}$, $\widetilde{X_N} = (\Gamma(N) \ltimes N \mathbb{Z}^d) \backslash \mathrm{ASL}_d(\mathbb{R})$ and denote by $\widetilde{\mu_N}$ the corresponding Haar measure.

When $\bm \alpha_\infty \in \mathbb{R}^d \setminus \mathbb{Q}^d$, using the same method as in the proof of \autoref{asladelicequid}, the result follows from this statement: 
for every $M_\infty \in \mathrm{SL}_d(\mathbb{R})$, every $\bm \alpha_\infty \in \mathbb{R}^d \setminus \mathbb{Q}^d$, every $N \in \mathbb{N}$ and every $g \in C_b(\widetilde{X_N})$,
\[ \lim_{t \to \infty} g((I_d, \bm \alpha)(M_\infty, \bm 0)(E(\bm x) \Phi^t, \bm 0)) d\lambda(\bm x) = \int_{\widetilde{X_N}} g d\mu_{\widetilde{X_N}}, \]
where $E(\bm x) = \exp \begin{pmatrix} 0 & \bm x \\ -\bm x^t & 0_{d-1} \end{pmatrix}$.
That statement is a special case of \cite[Corollary 5.4]{marklof_strombergsson_free_path_length_2010}.

When $\bm \alpha_\infty \in \mathbb{Q}^d$, it reduces to \autoref{sladelicequidrot} in the same way that \autoref{asladelicequid} reduces to \autoref{sladelicequid} in that case.
\end{proof}

\begin{cor} \label{limsupliminfaslrot} 
For every $M \in \mathrm{SL}_d(\mathbb{A})$, every Borel probability measure $\lambda$ on $S^{d-1}$ which is absolutely continuous with respect to the Lebesgue measure on $\mathbb{R}^{d-1}$, every measurable subset $\mathcal{E} \subset \mathrm{ASL}_d(\mathbb{Q}) \backslash \mathrm{ASL}_d(\mathbb{A})$ and every $\bm \alpha \in \mathbb{A}^d$, we have 
\begin{align*} & \limsup_{t \to \infty} \lambda( \{ \bm v \in S^{d-1} \, : \, (I_d, \bm \alpha) (M, \bm 0) \iota((K(\bm v) \Phi^t, \bm 0)) \in \mathcal{E} \}) \\ &\le  \begin{cases} \mu_{\mathrm{ASL}_d(\mathbb{Q}) \backslash \mathrm{ASL}_d(\mathbb{A})}(\overline{\mathcal{E}}) &\mbox{if } \bm \alpha_\infty \in \mathbb{R}^d \setminus \mathbb{Q}^d \\
\mu_{\mathrm{SL}_d(\mathbb{Q}) \backslash \mathrm{SL}_d(\mathbb{A})}(\{ g \in \mathrm{SL}_d(\mathbb{Q}) \backslash \mathrm{SL}_d(\mathbb{A}) \, : \, g (I_d, (\bm 0, \bm \beta M_f)) \in \overline{\mathcal{E}} \})  &\mbox{if } \bm \alpha_\infty \in \mathbb{Q}^d \end{cases} 
\end{align*}
and 
\begin{align*} &\liminf_{t \to \infty} \lambda( \{ \bm v \in S^{d-1} \, : \, (I_d, \bm \alpha) (M, \bm 0) \iota((K(\bm v) \Phi^t, \bm 0)) \in \mathcal{E} \}) \\
&\ge \begin{cases} \mu_{\mathrm{ASL}_d(\mathbb{Q}) \backslash \mathrm{ASL}_d(\mathbb{A})}(\accentset{\circ}{\mathcal{E}}) &\mbox{if } \bm \alpha_\infty \in \mathbb{R}^d \setminus \mathbb{Q}^d \\
\mu_{\mathrm{SL}_d(\mathbb{Q}) \backslash \mathrm{SL}_d(\mathbb{A})}(\{ g \in \mathrm{SL}_d(\mathbb{Q}) \backslash \mathrm{SL}_d(\mathbb{A}) \, : \, g (I_d, (\bm 0, \bm \beta M_f)) \in \accentset{\circ}{\mathcal{E}} \})  &\mbox{if } \bm \alpha_\infty \in \mathbb{Q}^d
, \end{cases} 
\end{align*}
where $\bm \beta = \bm \alpha_f - \delta(\bm \alpha_\infty)_f$.
\end{cor}

\section{Proof of the main theorem}

In this section, we build up to a proof of \autoref{mainthm}.

Define, for $\mathcal{A} \subset \mathbb{A}^d$ and a positive integer $r \ge 1$, \[ \mathcal{A}_{\ge r}^{\#} = \{ g \in \mathrm{SL}_d(\mathbb{Q}) \backslash \mathrm{SL}_d(\mathbb{A}) \, : \, \#(\mathcal{A} \cap \delta(\mathbb{Q})^d g \setminus \{ \mathbf 0\}) \ge r \}. \]
\begin{lemma}\label{countprop} The following properties hold:
\begin{enumerate}
\item If $\mathcal{A} \subset \mathcal{B} \subset \mathbb{A}^d$, then $\mathcal{A}_{\ge r}^{\#} \subset \mathcal{B}_{\ge r}^{\#}$.
\item If $\mathcal{A} \subset \mathbb{A}^d$ is open, then so is $\mathcal{A}_{\ge r}^{\#}$.
\item If $\mathcal{A} \subset \mathbb{A}^d$ is closed and bounded, then $\mathcal{A}_{\ge r}^{\#}$ is closed.
\item If $\mathcal{A} \subset \mathbb{A}^d$ is such that $m(\mathcal{A}) = 0$, then $\mu(\mathcal{A}_{\ge r}^{\#}) = 0$.
\end{enumerate}
\end{lemma}

Before proceeding with the proof which is, {\em mutatis mutandis}, the proof of \cite[Lemma 6.2]{marklof_strombergsson_free_path_length_2010}, we comment on the choice of a metric on $\mathbb{A}^d$. As we only need to know that there is a nice (translation-invariant, homogeneous) distance on $\mathbb{A}^d$ we could appeal to a general result on the existence of such metrics in locally compact second-countable topological groups due to Struble \cite{struble_metrics_1974}. However, in our special case, it is also possible to give an explicit definition of such a distance, as done by McFeat in \cite[Part I, 3.1]{McFeat1971} --- see also \cite{TorbaZuniga2013, HaynesMunday2013}.
With that out of the way and such a metric $\rho$ on $\mathbb{A}^d$, we set for $\bm x \in \mathbb{A}^d$, $\|\bm x\| = \rho(\bm 0, \bm x)$.

 \begin{proof} 
 
 1. This is immediate.
 
 2. Let $g_0 \in \mathcal{A}_{\ge r}^\#$. Find $\bm v_1, \ldots, \bm v_r \in \delta(\mathbb{Q})^d \setminus \{ \mathbf{0} \}$ such that for every $i, \, \bm v_i g_0 \in \mathcal{A}$.
 Define $f_i \colon \mathrm{SL}_d(\mathbb{A}) \to \mathbb{A}^d$, $g \mapsto \bm v_i g$ and $\Omega = \bigcap_{i = 1}^r f_i^{-1}(\mathcal{A})$. Observe that $g_0 \in \Omega$ and each point of $\Omega$ projects to a point of $\mathcal{A}_{\ge r}^\#$. 
 Since $\Omega$ is open (as the finite intersection of the $f_i^{-1}(\mathcal{A})$ which are all open because $\mathcal{A}$ is open and every $f_i$ is continuous), this implies that $\mathcal{A}_{\ge r}^\#$ is open.  
 
 3. Let $g_0 \in X \setminus \mathcal{A}_{\ge r}^\#$. This means $\#(\mathcal{A} \cap \delta(\mathbb{Q})^d g_0 \setminus \{ \mathbf 0\})  < r$. 
 Let $U$ be a neighbourhood of the identity in $\mathrm{SL}_d(\mathbb{A})$ such that for every $\bm x \in \mathbb{A}^d$ and every $M \in U$, we have \begin{equation} \| \bm xM - \bm x \| \le \frac 12 \| \bm x \|. \end{equation}
 By assumption $\mathcal{A}$ is bounded, so we can define $R = \max_{\bm x \in \mathcal{A}} \| \bm x \|$.
 For any $g \in U$ and any $\bm x \in \mathbb{A}^d$ satisfying $\| \bm x \| > 2R$, it follows that $\bm xg \notin \mathcal{A}$. Indeed: 
 \begin{equation} \|\bm xg\| \ge \| \bm x \| - \|\bm xg - \bm x\| \ge \frac 12 \|\bm x\| > R. \end{equation}
 Define the finite set $F = \{ \vecr \in \delta(\mathbb{Q})^d \setminus \{ \mathbf{0} \} \, : \, \| \vecr  g_0 \| \le 2R \text{ and } \vecr g \notin \mathcal{A} \}$.
 For each $\vecr \in F$, choose an open set $V_\vecr \subset \mathbb{A}^d$ such that $\vecr g_0 \in V_\vecr \subset \mathbb{A}^d \setminus \mathcal{A}$.
 Set $U' = Ug_0 \cap \left( \bigcap_{\vecr \in F} \{ g \in X \, : \, \vecr g \in V_\vecr \} \right)$. 
 $U'$ is an open subset of $X$ and $g_0 \in U'$. Moreover, by construction each $g \in U'$ projects to a point of $X \setminus \mathcal{A}_{\ge r}^\#$, which concludes the proof that $\mathcal{A}_{\ge r}^\#$ is closed.
 
 4. We have, for $r \ge 1$, 
 \begin{align} \mu(\mathcal{A}_{\ge r}^\#) &\le \mu(\mathcal{A}_{\ge 1}^\#) \\
 &= \mu(\{ g \in X \, : \, \mathcal{A} \cap \delta(\mathbb{Q})^d g \ne \{ \bm 0 \} \}) \\
 &\le \sum_{\bm q \in \delta(\mathbb{Q})^d \setminus \{\mathbf{0}\}} \mu(\{ g \in X \, : \,  \bm q g \in \mathcal{A} \}) \\
 &= \sum_{\bm q \in \delta(\mathbb{Q})^d \setminus \{\mathbf{0}\}} \int_X \chi_\mathcal{A}(\bm qg) d\mu(g) \\
 &= \int_X \sum_{\bm q \in \delta(\mathbb{Q})^d \setminus \{\mathbf{0}\}} \chi_\mathcal{A}(\bm qg) d\mu(g) \\
 &= m(\mathcal{A}) \\
 &= 0, \end{align}
 where the penultimate equality follows from \autoref{Siegelformula}.
 \end{proof}
 
 \begin{thm}\label{slrandhoro}
Let $\mathcal{A} \subset \mathbb{A}^d$ be a bounded Borel set with $m(\partial \mathcal{A}) = 0$. 
For every $M \in \mathrm{SL}_d(\mathbb{A})$, every $r \in \mathbb{Z}_{\ge 0}$ and every Borel probability measure $\lambda \ll \operatorname{Leb}_{d-1}$, \begin{equation} \lambda(\{ \bm x \in \mathbb{R}^{d-1} \, : \, \#(\mathcal{A} \iota(\Phi^{-t} n(\bm x)) \cap \delta(\mathbb{Q})^d M \setminus \{ \mathbf{0} \}) \ge r \}) \end{equation} has a limit when $t$ tends to infinity and it is given by 
\begin{equation} \mu( \{ g \in \mathrm{SL}_d(\mathbb{Q}) \backslash \mathrm{SL}_d(\mathbb{A}) \, : \, \# (\mathcal{A} \cap \delta(\mathbb{Q})^d g \setminus \{ \mathbf{0} \}) \ge r \} ). \end{equation}
\end{thm}

 \begin{proof} Consider $\overline{\mathcal{A}}_{\ge r}^\#$, which is closed by property 3 of \autoref{countprop}, and observe: 
 \begin{align} \label{upper}
 &\limsup_{t \to +\infty} \lambda(\{ \bm x \in \mathbb{R}^{d-1} \, : \, \# (\mathcal{A} \iota( \Phi^{-t} n(\bm x) ) \cap \delta(\mathbb{Q})^d M \setminus \{ \mathbf{0} \}) \ge r \}) \\
 &\le \limsup_{t \to +\infty} \lambda(\{ \bm x \in \mathbb{R}^{d-1} \, : \, \# (\overline{\mathcal{A}} \cap \delta(\mathbb{Q})^d M  \iota(n(-\bm x) \Phi^t) \setminus \{ \mathbf{0} \}) \ge r \}) \\
 &\le \mu(\overline{\mathcal{A}}_{\ge r}^\#). \end{align}
 The first inequality follows from property 1 of \autoref{countprop} and the second one from \autoref{limsupliminf}.

 Similarly, $\accentset{\circ}{\mathcal{A}}_{\ge r}^\#$ is open by property 2 of \autoref{countprop} and we get 
 \begin{align} \label{lower}
 &\liminf_{t \to +\infty} \lambda(\{ \bm x \in \mathbb{R}^{d-1} \, : \, \# (\mathcal{A} \iota( \Phi^{-t} n(\bm x)) \cap \delta(\mathbb{Q})^d M \setminus \{ \mathbf{0} \}) \ge r \}) \\
 &\ge \mu(\accentset{\circ}{\mathcal{A}}_{\ge r}^\#). \end{align}
 
Finally, note that \[ \overline{\mathcal{A}}_{\ge r}^\# \setminus \accentset{\circ}{\mathcal{A}}_{\ge r}^\# \subset (\partial A)_{\ge 1}^\# \] so that property 4 of \autoref{countprop} implies $\mu((\partial A)_{\ge 1}^\#) = 0$. 
Hence $\mu(\overline{\mathcal{A}}_{\ge r}^\#) = \mu(\accentset{\circ}{\mathcal{A}}_{\ge r}^\#)$, which yields the desired conclusion in view of \eqref{upper} and \eqref{lower}. \end{proof}

We also define, for $\mathcal{A} \subset \mathbb{A}^d$ and an integer $r \ge 1$, 
\[ [\mathcal{A}]_{\ge r} = \{ g \in \mathrm{ASL}_d(\mathbb{Q}) \backslash \mathrm{ASL}_d(\mathbb{A}) \, : \, \#(\mathcal{A} \cap \delta(\mathbb{Q})^d g) \ge r \} \] and prove the analogue of \autoref{countprop} for these sets, which in turn yields an analogue of \autoref{slrandhoro} for the special affine group.

\begin{lemma} \label{countpropasl} The following properties hold:
\begin{enumerate}
\item If $\mathcal{A} \subset \mathcal{B} \subset \mathbb{A}^d$, then $[\mathcal{A}]_{\ge r} \subset [\mathcal{B}]_{\ge r}$.
\item If $\mathcal{A} \subset \mathbb{A}^d$ is open, then so is $[\mathcal{A}]_{\ge r}$.
\item If $\mathcal{A} \subset \mathbb{A}^d$ is closed and bounded, then $[\mathcal{A}]_{\ge r}$ is closed.
\item If $\mathcal{A} \subset \mathbb{A}^d$ is such that $m(\mathcal{A}) = 0$, then $\mu_{\mathrm{ASL}_d(\mathbb{Q}) \backslash \mathrm{ASL}_d(\mathbb{A})}([\mathcal{A}]_{\ge r}) = 0$.
\end{enumerate}
\end{lemma}

\begin{proof}
(i) is once again immediate and (ii) can be readily proved by adapting the proof of (ii) in \autoref{countprop}.

For (iii), the only change to be made is that $2R$ now needs to be $4R$ in order to be able to write, for $g = (M, \bm \xi) \in \mathrm{ASL}_d(\mathbb{A})$ in the appropriate neighbourhood of the identity and $\bm x \in \mathbb{A}^d$ with $\| x \| > 4R$, 
\[ \| \bm x g \| = \| \bm x M + \bm \xi \| \ge \| \bm x \| - \| \bm x M - M \| - \| \bm \xi \| > \frac 12 \| \bm x \| - R > R. \]

As for (iv), this time we have 
\begin{align*}
&\mu_{\mathrm{ASL}_d(\mathbb{Q}) \backslash \mathrm{ASL}_d(\mathbb{A})}(\{ g \in \mathrm{ASL}_d(\mathbb{A}) \, : \, \mathcal{A} \cap \delta(\mathbb{Q})^d g \ne \emptyset \}) \\
&\le \sum_{\bm q \in \delta(\mathbb{Q})^d} \int_{\mathrm{SL}_d(\mathbb{A})} \int_{\mathbb{A}^d} \chi_{\mathcal{A}}(\bm q M + \bm \xi) d \bm \xi d\mu_{\mathrm{SL}_d(\mathbb{Q}) \backslash \mathrm{SL}_d(\mathbb{A})}(M)
\end{align*}
and the statement once again follows from \autoref{Siegelformula}.
\end{proof}

\begin{thm}\label{aslrandhoro}
Let $\mathcal{A} \subset \mathbb{A}^d$ be a bounded Borel set with $m(\partial \mathcal{A}) = 0$. 
For every $M \in \mathrm{SL}_d(\mathbb{A})$, every $\bm \alpha \in \mathbb{R}^d$, every $r \in \mathbb{Z}_{\ge 0}$ and every Borel probability measure $\lambda \ll \operatorname{Leb}_{d-1}$, \begin{equation} \lambda(\{ \bm x \in \mathbb{R}^{d-1} \, : \, \#(\mathcal{A} \iota((\Phi^{-t} n(\bm x), \bm 0)) (I_d, \bm \alpha) \cap \delta(\mathbb{Q})^d M) \ge r \}) \end{equation} has a limit when $t$ tends to infinity and it is given by 
\begin{equation} \mu_{\mathrm{ASL}_d(\mathbb{Q}) \backslash \mathrm{ASL}_d(\mathbb{A})}( \{ g \in \mathrm{ASL}_d(\mathbb{Q}) \backslash \mathrm{ASL}_d(\mathbb{A}) \, : \, \# (\mathcal{A} \cap \delta(\mathbb{Q})^d g) \ge r \} )\end{equation}
if $\bm \alpha \in \mathbb{R}^d \setminus \mathbb{Q}^d$
and by
\begin{equation} \mu_{\mathrm{SL}_d(\mathbb{Q}) \backslash \mathrm{SL}_d(\mathbb{A})}( \{ g \in \mathrm{SL}_d(\mathbb{Q}) \backslash \mathrm{SL}_d(\mathbb{A}) \, : \, \# (\mathcal{A} - (\bm 0, \bm \beta M_f) \cap \delta(\mathbb{Q})^d g \setminus \{ \bm 0 \} ) \ge r \} )\end{equation}
if $\bm \alpha \in \mathbb{Q}^d$,
where $\bm \beta = - \delta(\bm \alpha)_f$.
\end{thm}

\begin{proof}
The proof follows the same lines as the proof of \autoref{slrandhoro}, using \autoref{countpropasl} instead of \autoref{countprop} and \autoref{limsupliminfasl} instead of \autoref{limsupliminf}.
\end{proof}

We can now see that \autoref{mainthm} is proved just like \autoref{aslrandhoro} thanks this time to \autoref{limsupliminfaslrot} instead of \autoref{limsupliminfasl}.

\section{Applications} \label{apps}

For $d \ge 2$, a lattice $\mathcal{L}$ in $\mathbb{A}^d$ and a bounded ``window set'' $\mathcal{W} \subset \mathbb{A}_f^d$ 
, define the (adelic) cut-and-project set
\[ \mathcal{P}(\mathcal{L}, \mathcal{W}) = \{ \pi(\bm x) \, : \, \bm x \in \mathcal L, \, \pi_{\text{int}}(\bm x) \in \mathcal W \} \subset \mathbb{R}^d, \]
where $\pi \colon \mathbb{A}^d \to \mathbb{R}^d$ and $\pi_{\text{int}} \colon \mathbb{A}^d \to \mathbb{A}_f^d$ are the canonical projections.

In this section, a crucial assumption is $m_f(\partial \mathcal{W})=0$.

A key feature is that we are able to extend the results presented below to point sets such as the set of primitive lattice points in $\mathbb{R}^d$ as discussed in the introduction. It does fall within our framework and can be viewed as $\mathcal{P}(\mathbb{Q}^d, \mathcal{W})$, with however $m_f(\partial \mathcal{W}) > 0$. We discuss this and other examples in \autoref{examples}. That our results carry through for such ``irregular'' point sets follows from a probabilistic result we prove in \autoref{densityincr}, and which may be of independent interest.
We discuss this extension in \autoref{extension}.

\subsection{Free path length}\label{appfpl}
Recall that for a set of scatterers with common radius $\rho$, the free path length $\tau(\bm q, \bm v, \rho)$ is defined (see \eqref{fpldef}) as the first time a particle with initial position $\bm q \in \mathbb{R}^d$ and initial direction $\bm v \in S^{d-1}$ hits one of the scatterers.
Define, for $\xi > 0$, the cylinder \[ \mathscr{C}_{\xi} = \{ (x_1, x_2 \ldots, x_d) \in \mathbb{R}^d \, : \, 0 < x_1 < \xi, \, x_2^2 + \ldots + x_d^2 < 1  \}. \]

\begin{cor}\label{fpl} Fix a lattice $\mathcal{L} = \mathbb{Q}^d A$ in $\mathbb{A}^d$, $A \in \mathrm{SL}_d(\mathbb{A})$.  Let $\bm q$ in $\mathbb{R}^d$ and $\bm \alpha = -\bm qA_\infty^{-1}$. For every bounded $\mathcal{W} \subset \mathbb{A}_f^d$ with $m_f(\partial \mathcal{W})=0$, there exists a function $F_{\mathcal{P}, \bm \alpha} \colon \mathbb{R}_{\ge 0} \to [0,1]$ such that for every Borel probability measure $\lambda$ on $S^{d-1}$ which is absolutely continuous with respect to the Lebesgue measure and every $\xi > 0$, 
\begin{equation} \lim_{\rho \to 0} \lambda(\{ \bm v \in S^{d-1} \, : \, \rho^{d-1} \tau(\bm q, \bm v, \rho) \ge \xi \}) = F_{\mathcal{P}, \bm \alpha}(\xi).\end{equation} 
Moreover, for $\bm \alpha \in \mathbb{R}^d \setminus \mathbb{Q}^d$, the limiting distribution $F_{\mathcal{P},\bm \alpha}$ is independent of $\bm \alpha$ (and thus $\bm q$) and given by 
\begin{align} \label{fpllimdistirr}
\begin{split}
 F_\mathcal{P}(\xi) &= \mu(\{M \in \mathrm{ASL}_d(\mathbb{Q}) \backslash \mathrm{ASL}_d(\mathbb{A}) \, : \,  \mathcal{P}(\delta(\mathbb{Q})^d M, \mathcal{W}) \cap \mathscr{C}_\xi = \emptyset \}) \\
&= \mu(\{M \in \mathrm{ASL}_d(\mathbb{Q}) \backslash \mathrm{ASL}_d(\mathbb{A}) \, : \,  \delta(\mathbb{Q})^d M \cap (\mathscr{C}_\xi \times \mathcal{W}) = \emptyset \}).  
\end{split}
\end{align}
For $\bm \alpha \in \mathbb{Q}^d$, the limiting distribution $F_{\mathcal{P}, \bm \alpha}$ is given by 
\begin{align} \label{fpllimdistrat}
\begin{split} F_\mathcal{P, \bm \alpha}(\xi) &= \mu(\{M \in \mathrm{SL}_d(\mathbb{Q}) \backslash \mathrm{SL}_d(\mathbb{A}) \, : \,  \mathcal{P}(\delta(\mathbb{Q})^d M, \mathcal{W}_{\bm \beta}) \cap \mathscr{C}_\xi = \emptyset \}) \\
&= \mu(\{M \in \mathrm{SL}_d(\mathbb{Q}) \backslash \mathrm{SL}_d(\mathbb{A}) \, : \,  \delta(\mathbb{Q})^d M \cap (\mathscr{C}_\xi \times \mathcal{W}_{\bm \beta}) = \emptyset \}),  
\end{split}
\end{align}
where $\bm \beta \in \delta(\mathbb{Q})^d$ is such that $\bm \alpha = \pi(\bm \beta)$ and $\mathcal{W}_{\bm \beta} = \mathcal{W} - \pi_{\text{int}}(\bm \beta)A_f$.
\end{cor}

\begin{proof}
This essentially follows from \autoref{mainthm} applied to the set $\mathcal{A} = \mathscr{C}_\xi \times \mathcal{W}$.

Our assumption that $m_f(\partial \mathcal{W})=0$ allows us to verify that $m(\partial \mathcal{A})=0$. 

To see why this implies \autoref{fpl}, notice that \begin{align}
\label{longcyl} 
\lambda(\{ \bm v \in S^{d-1} \, : \, \widetilde{\mathscr{C}} \cap \mathcal{P} = \emptyset \}) &\leq \lambda(\{ \bm v \in S^{d-1} \, : \, \rho^{d-1} \tau(\bm q, \bm v; \rho) \ge \xi \}) \\
\label{shortcyl}
 &\leq \lambda(\{ \bm v \in S^{d-1} \, : \, \underline{\mathscr{C}} \cap \mathcal{P} = \emptyset \}) \end{align}
where $\underline{\mathscr{C}}$ is the open cylinder of radius $\rho$ about the line segment from $\bm q$ to $\bm q+(\rho^{1-d} \xi - \rho) \bm v$ and $\widetilde{\mathscr{C}}$ is the open cylinder of radius $\rho$ about the line segment from $\bm q$ to $\bm q+(\rho^{1-d}\xi + \rho)\bm v$.
In fact it is convenient to replace, for $\varepsilon > 0$ and $\rho$ small enough, $\underline{\mathscr{C}}$ by the slightly shorter open cylinder of radius $\rho$ about the line segment from $\bm q$ to $\bm q+\rho^{1-d} (\xi - \varepsilon) \bm v$ and $\widetilde{\mathscr{C}}$ by the slightly longer open cylinder of radius $\rho$ about the line segment from $\bm q$ to $\bm q+\rho^{1-d} (\xi + \varepsilon)\bm v$. 
We still denote those cylinders by $\underline{\mathscr{C}}$ and $\widetilde{\mathscr{C}}$ respectively, so the above upper and lower bound stay exactly the same.
The point is that we can write \[ \underline{\mathscr{C}} = \mathscr{C}_{\xi-\varepsilon} \Phi^{\log \rho} K(\bm v)^{-1}(I_d, \bm q) \] and similarly \[ \widetilde{\mathscr{C}} = \mathscr{C}_{\xi+\varepsilon} \Phi^{\log \rho} K(\bm v)^{-1}(I_d, \bm q). \]
Hence \[ \underline{\mathscr{C}} \cap \mathcal{P} = \emptyset \iff (\mathscr{C}_{\xi-\varepsilon} \times \mathcal{W}) \cap \delta(\mathbb{Q})^d A \iota((I_d,-\bm q)K(\bm v)\Phi^{-\log \rho}) = \emptyset \] and similarly for $\widetilde{\mathscr{C}}$.

At this point we can indeed apply our main theorem (\autoref{mainthm}) to get that \eqref{shortcyl} converges,  as $\rho \to 0$,  to the expression for $F_{\mathcal{P}, \bm \alpha}(\xi - \varepsilon)$ given by \eqref{fpllimdistirr} and \eqref{fpllimdistrat} while the left-hand side of \eqref{longcyl} converges to $F_{\mathcal{P}, \bm \alpha}(\xi + \varepsilon)$.
Using the continuity of the expressions \eqref{fpllimdistirr} and \eqref{fpllimdistrat} (see \autoref{propslimdistr}), we can let $\varepsilon$ go to $0$ to get the result.

\end{proof}

\subsection{Local statistics of directions} \label{directions}
For $\bm \alpha \in \mathbb{R}^d$, consider the set of nonzero points of $\mathcal{P}+\bm \alpha$ inside the open $d$-ball of radius $T>0$ centred at the origin, or more generally, for $0 \le c < 1$, 
\begin{equation} \mathcal{P}_T(\bm \alpha, c) = ((\mathcal{P} + \bm \alpha) \cap \{\bm x \in \mathbb{R}^d \, : \, cT \le \| \bm x \|_2 < T\})  \setminus \{ \mathbf{0} \}. \end{equation}
Provided $m_f(\partial \mathcal{W})=0$ we have, for every $\bm \alpha \in \mathbb{R}^d$ and every $0 \le c < 1$, \begin{equation} \lim_{T \to \infty} \frac{\# \mathcal{P}_T(\bm \alpha, c)}{T^d} = (1-c^d) m_f(\mathcal{W}) \operatorname{vol}(\mathcal{B}^d) \end{equation}
where $m_f$ is the Haar measure on $\mathbb{A}_f^d$.
For a fixed $c \in [0,1)$ and each $T > 0$, as $\bm y$ ranges through $\mathcal{P}_T(\bm \alpha, c)$, we are interested in the distribution of the directions $\dfrac {\bm y}{\| \bm y \|_2} \in S^{d-1}$, counted with multiplicity.

The above counting asymptotic shows that the set of directions is uniformly distributed over $S^{d-1}$ as $T \to \infty$.
To understand the local statistics of the directions to points in $\mathcal{P}_T(\bm \alpha, c)$, it is enough to look at the probability of finding $r$ directions in a small open disc $\mathcal{D}_T(c, \sigma, \bm v) \subset S^{d-1}$ with random centre $\bm v \in S^{d-1}$ and radius chosen so that the area of the disc is equal to $\frac {\sigma d}{(1- c^d) m_f(\mathcal{W}) T^d}$ with $\sigma > 0$ fixed.
This way, defining the random variable \begin{equation} \mathcal{N}_T(c, \sigma, \bm v) = \# \left\{\bm y \in \mathcal{P}_T(\bm \alpha, c) \, : \, \frac {\bm y}{\|\bm y\|_2} \in \mathcal{D}_T(c, \sigma, \bm v) \right\}, \end{equation}
we have $\mathbb{E}(\mathcal{N}_T(c, \sigma, \bm v)) \xrightarrow[T \to \infty]{} \sigma$.

Define, for $c \in [0, 1)$ and $\sigma > 0$, the cone
\begin{align*}
 \mathscr{K}_{c, \sigma} = \Bigg\{ (x_1, \ldots, x_d) &\in \mathbb{R}^d \, : \, c < x_1 < 1, \\
  & \hspace{-0.2cm} \|(x_2, \ldots, x_d) \|_2 \le \left( \frac {\sigma d}{(1-c^d) m_f(\mathcal{W}) \operatorname{vol}(\mathcal{B}^{d-1})} \right)^{1/(d-1)} x_1 \Bigg\}. 
 \end{align*}

\begin{cor}\label{dir} For every lattice $\mathcal{L} = \mathbb{Q}^d A$ in $\mathbb{A}^d$, $A \in \mathrm{SL}_d(\mathbb{A}_f)$, every bounded $\mathcal{W} \subset \mathbb{A}_f^d$ with $m_f(\partial \mathcal{W})=0$, every $\bm \alpha$ in $\mathbb{R}^d$, every $c \in [0, 1)$, there exists a function $G_{\mathcal{P}, \bm \alpha, c} \colon \mathbb{R}_{\ge 0} \times \mathbb{Z}_{\ge 0} \to [0,1]$ such that for every Borel probability measure $\lambda$ on $S^{d-1}$ which is absolutely continuous with respect to the Lebesgue measure, every $r \in \mathbb{Z}_{\ge 0}$ and every $\sigma > 0$,
\begin{equation} \lim_{T \to \infty} \lambda(\{ \bm v \in S^{d-1} \, : \, \mathcal{N}_T(c, \sigma, \bm v) = r \}) = G_{\mathcal{P}, \bm \alpha, c}(\sigma, r). \end{equation}
Moreover, for $\bm \alpha \in \mathbb{R}^d \setminus \mathbb{Q}^d$, the limiting distribution $G_{\mathcal{P}, \bm \alpha, c}$ is independent of $\bm \alpha$ and given by
\[ G_{\mathcal{P}, c}(\sigma, r) = \mu(\{M \in \mathrm{ASL}_d(\mathbb{Q}) \backslash \mathrm{ASL}_d(\mathbb{A}) \, : \, \#(\mathcal{P}(\delta(\mathbb{Q})^d M, \mathcal{W}) \cap \mathscr{K}_{c, \sigma}) = r \}). \]
For $\bm \alpha \in \mathbb{Q}^d$, the limiting distribution $G_{\mathcal{P}, \bm \alpha, c}$ is given by 
\[ G_{\mathcal{P}, \bm \alpha, c}(\sigma, r) = \mu(\{M \in \mathrm{SL}_d(\mathbb{Q}) \backslash \mathrm{SL}_d(\mathbb{A}) \, : \, \#(\mathcal{P}(\delta(\mathbb{Q})^d M, \mathcal{W}_{\bm \beta}) \cap \mathscr{K}_{c, \sigma}) = r \}) \]
where $\bm \beta \in \delta(\mathbb{Q})^d$ is such that $\bm \alpha = \pi(\bm \beta)$ and $\mathcal{W}_{\bm \beta} = \mathcal{W} - \pi_{\text{int}}(\bm \beta)A_f$.
 \end{cor}

\begin{proof} One can mimic the proof of \autoref{fpl} using the cone $\mathscr{K}_{c, \sigma}$ instead of the cylinder $\mathscr{C}_{\xi}$. \end{proof}

\begin{rmk} As mentioned in \autoref{intro}, \autoref{dir} applied to $d=2$ and $r=0$ immediately implies the existence of a limiting gap distribution (defined by \eqref{defgapdist}) as $T \to \infty$ for the sequence of angles $(\alpha_j(T))$ produced by radial projection of our point set. \end{rmk}

\section{Examples of adelic model sets}\label{examples}

We give a few examples of weak adelic model sets in this section. They seem to originate in Meyer's paper \cite{Meyeradeles} and their diffraction patterns were studied by Baake, Moody and Pleasants \cite{BaakeMoodyPleasants}. 

\subsection{Primitive lattice points}
By taking the window $\mathcal{W} = \prod_p (\mathbb{Z}_p^d \setminus p \mathbb{Z}_p^d) \subset \mathbb{A}_f^d$ (and any lattice in $\mathbb{A}^d$) in the cut-and-project construction, we obtain the primitive lattice points (of the corresponding lattice in $\mathbb{R}^d$).
For instance, with the adelic lattice $\mathbb{Q}^d$, we obtain the visible lattice points.

Note that $\accentset{\circ}{\mathcal{W}} = \emptyset$. Indeed, $\widehat{\mathbb{Z}}^d$ is endowed with the product topology so if $\mathcal{W}$ were to contain a non-empty open set, it would contain a basic open set of the form $\prod_{p \in \mathbb{P}} U_p$ where each $U_p$ is an open subset of $\mathbb{Z}_p^d$ and all but finitely many of these are actually equal to $\mathbb{Z}_p^d$. It is however apparent that $\mathcal{W}$ cannot contain such a basic open set and so must have empty interior.

\subsection{k-free coordinates}
We can also take $\prod_{p \in \mathbb{P}} (\mathbb{Z}_p \setminus p^k \mathbb{Z}_p)$ for some integer $k \ge 2$ as our basic building block and use it for one/all coordinates and possibly different values of $k$ for different coordinates to define our window set. We thus obtain lattice points one/all of whose coordinates are $k$-free with possibly different values of $k$ at each coordinate.

Again, all of those windows have empty interior.

\begin{figure}[!ht]
\begin{centering}\includegraphics[width=0.8\textwidth]{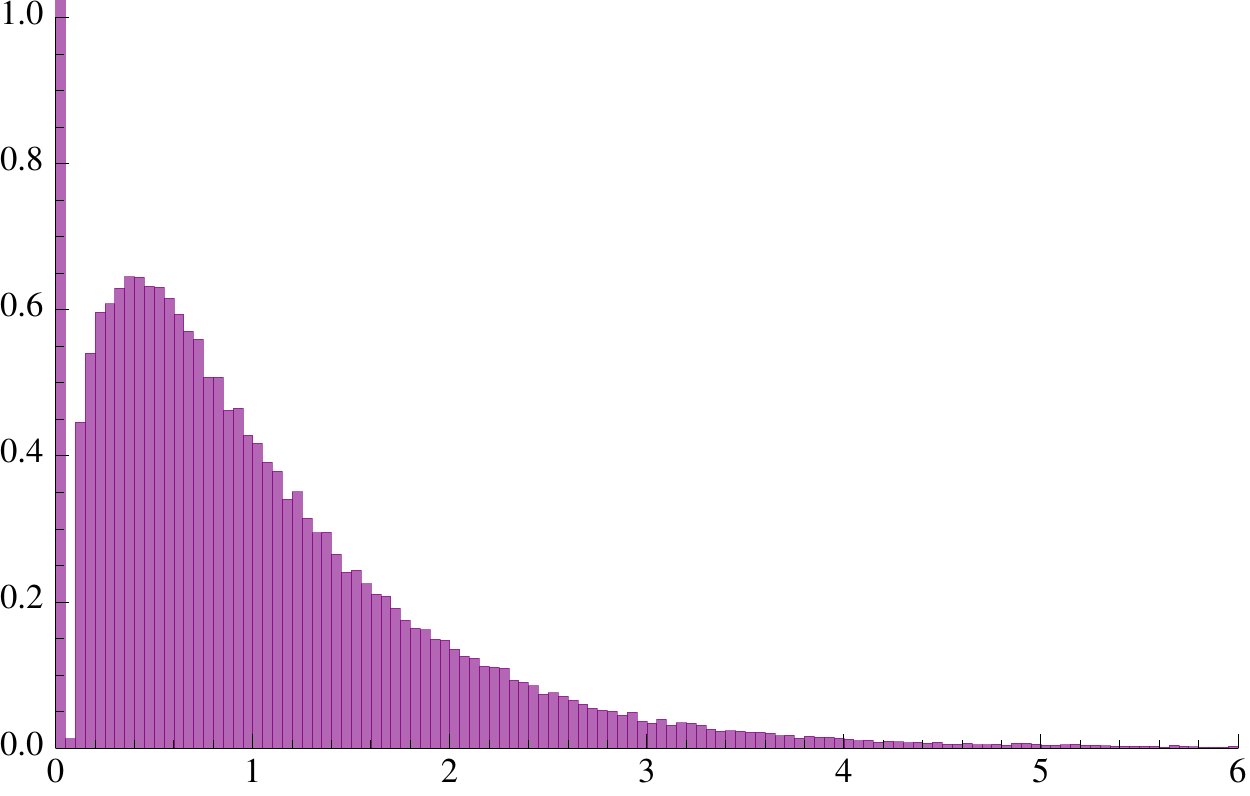}\par\end{centering}
\caption{The figure shows the gap distribution in the sequence obtained after radial projection of the lattice points both of whose coordinates are squarefree, shifted by $\bm \alpha =(\frac 12,\frac 12)$, $T=500$.}
\end{figure}

\begin{figure}[!ht]
\begin{centering}\includegraphics[width=0.8\textwidth]{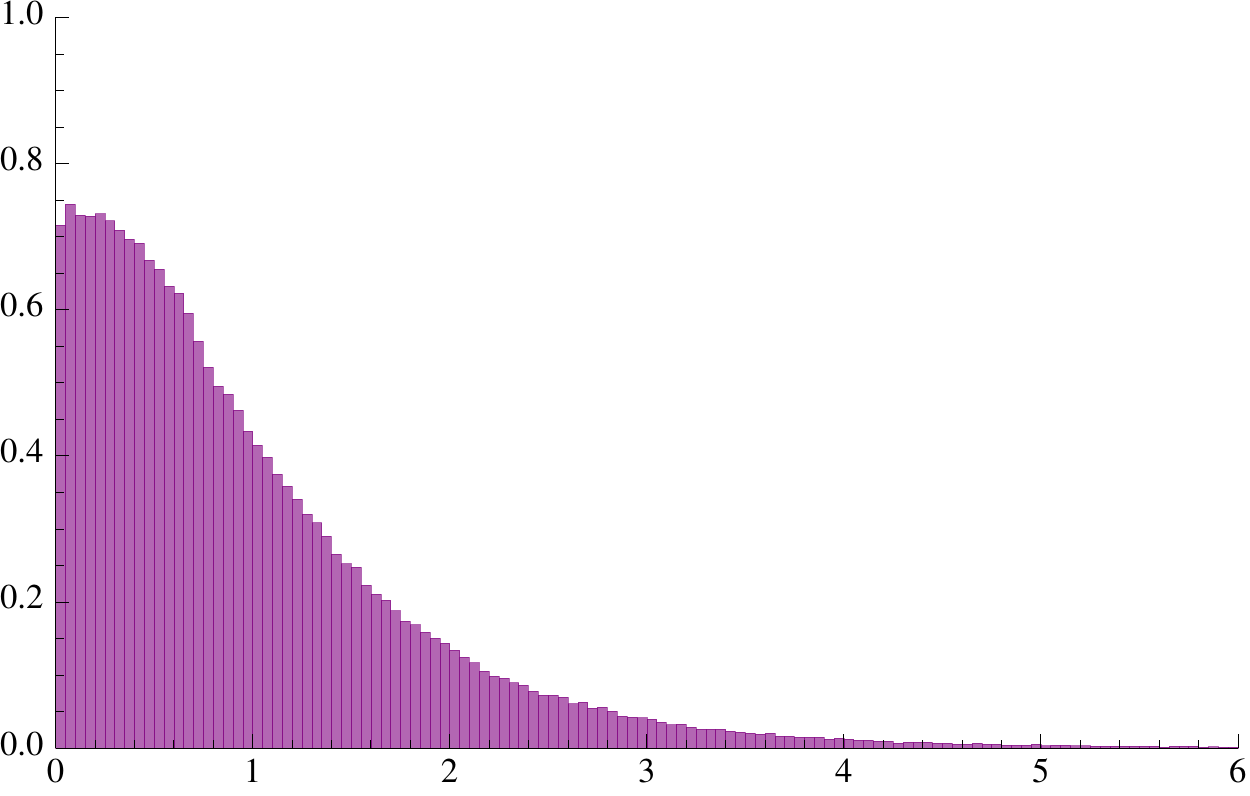}\par\end{centering}
\caption{The figure shows the gap distribution in the sequence obtained after radial projection of the lattice points both of whose coordinates are squarefree, shifted by $\bm \alpha =(\sqrt{2},\sqrt{3})$, $T=500$.}
\end{figure}



\section{A key probabilistic result}\label{densityincr}

Throughout this section, we consider a set of points $\mathcal{P} \subset \mathbb{R}^d$ and, for each $\varepsilon > 0$, a set of points $\mathcal{P}^\varepsilon$ such that $\mathcal{P} \subset \mathcal{P}^\varepsilon \subset \mathbb{R}^d$.

For $T>0$, we denote by $\mathcal{B}_T$ the ball of radius $T$ in $\mathbb{R}^d$ centred at the origin.
Assume that there exists a $c > 0$ such that these point sets satisfy \begin{equation}\label{asydens} \lim_{T \to +\infty} \frac {\#(\mathcal{P} \cap \mathcal{B}_T)}{\operatorname{vol}(\mathcal{B}_T)} = c  \end{equation} and \begin{equation}\label{upperasydens} \forall \varepsilon > 0, \exists T_0 > 0 \, | \, \forall T \ge T_0, \, \#((\mathcal{P}^\varepsilon  \setminus \mathcal{P}) \cap \mathcal{B}_T) \le \varepsilon \operatorname{vol}(\mathcal{B}_T). \end{equation}

Define, for a fixed bounded Jordan measurable set $\mathcal{A} \subset \mathbb{R}^d$, $t > 0$ and $\bm v \in S^{d-1}$ distributed according to a Borel probability measure $\lambda$ which is absolutely continuous with respect to the Lebesgue measure on the sphere, the random variables \[ \mathcal{N}_t = \#(\mathcal{A} \Phi^{-t} K(\bm v)  \cap \mathcal{P}) \] and \[ \mathcal{N}_t^\varepsilon = \#( \mathcal{A} \Phi^{-t} K(\bm v) \cap \mathcal{P}^\varepsilon). \]

\begin{lemma} \label{boundedexpectation}
For every bounded Jordan measurable set $\mathcal{A} \subset \mathbb{R}^d$, there exists a constant $C_\mathcal{A} \ge 0$ such that for every $t > 0$,
\begin{equation} \int_{S^{d-1}} \mathcal{N}_t d\Leb \le C_\mathcal{A}. \end{equation}
\end{lemma}

\begin{proof}
We follow the calculation performed in \cite[Proof of Theorem 5.1]{marklof_strombergsson_quasicrystals_fpl_2014}.
We write \begin{equation} \int_{S^{d-1}} \mathcal{N}_t d\Leb = \int_{S^{d-1}} \sum_{x \in \mathcal{P} K(\bm v)^{-1} \Phi^t} \chi_\mathcal{A}(x) d\Leb. \end{equation}
Since $\mathcal{A}$ is bounded, we can find $R_\mathcal{A} > 0$ such that $\mathcal{A} \subset \mathcal{B}_{R_\mathcal{A}}$. Thus
\begin{align}
\int_{S^{d-1}} \mathcal{N}_t d\Leb &\le \int_{S^{d-1}} \sum_{x \in \mathcal{P} K(\bm v)^{-1} \Phi^t} \chi_{\mathcal{B}_{R_\mathcal{A}}}(x) d\Leb \\  
&= \sum_{\bm x \in \mathcal{P}} \int_{S^{d-1}} \chi_{\mathcal{B}_{R_\mathcal{A}} \Phi^{-t} K(\bm v)}(x) d\Leb.
\end{align}
Now for every $t > 0$ and $\bm x \in \mathcal{P} \setminus \{ \bm 0 \}$, we have 
\begin{equation}
\frac 1{\Leb(S^{d-1})} \int_{S^{d-1}} \chi_{\mathcal{B}_{R_\mathcal{A}} \Phi^{-t} K(\bm v)}(\bm x) d\Leb = A_t \left(\frac {\| \bm x \|}{R_\mathcal{A}} \right	),
\end{equation}
where for $\tau > 0$, 
\begin{equation} A_t(\tau) = \frac {\Leb(S^{d-1} \cap \tau^{-1} \mathcal{B}_1 \Phi^{-t})}{\Leb(S^{d-1})}. \end{equation}
As in \cite[Proof of Theorem 5.1]{marklof_strombergsson_quasicrystals_fpl_2014} we can now rewrite \[ \sum_{\bm x \in \mathcal{P} \setminus \{ \bm 0 \}} \int_{S^{d-1}} \chi_{\mathcal{B}_{R_\mathcal{A}} \Phi^{-t} K(\bm v)}(\bm x) d\Leb \] as a Riemann--Stieltjes integral with respect to $-A_t$ which is continuous and increasing. We can then use \eqref{asydens} instead of \cite[Proposition 3.2]{marklof_strombergsson_quasicrystals_fpl_2014} to reproduce the asymptotic estimates obtained by Marklof and Str\"ombergsson and deduce that 
\begin{equation} \lim_{t \to +\infty} \frac 1{\Leb(S^{d-1})} \int_{S^{d-1}} \mathcal{N}_t d\Leb \le \vol(B_{R_{\mathcal{A}}}). \end{equation}
The statement of the lemma follows.
\end{proof}

\begin{thm}\label{proba}
If, for every $\varepsilon > 0$, $\mathcal{P}$ and $\mathcal{P}^\varepsilon$ are as above and $(\mathcal{N}_t^\varepsilon)_{t >0}$ converges in distribution (to $\mathcal{N}^\varepsilon$), then so does $(\mathcal{N}_t)_{t>0}$ (to $\mathcal{N}$).
Furthermore, $\mathcal{N}^\varepsilon \xrightarrow[\varepsilon \to 0]{d} \mathcal{N}$.
\end{thm}

\begin{proof} Up to rescaling, we assume that the constant $c$ in \eqref{asydens} 
is $1$.
We begin by showing that $(\mathcal{N}_t)_{t>0}$ is tight. 
By \autoref{boundedexpectation}, there exists a constant $C_\mathcal{A}$ such that for every $t > 0$, \begin{equation}\label{convNt} \int_{S^{d-1}} N_t d\Leb \le C_\mathcal{A} \end{equation}
For every $M > 0$ and every $t > 0$, we want to get an upper bound for \[ \lambda(\{ \bm v \in S^{d-1} : \mathcal{N}_t \ge M \}) = \int_{S^{d-1}} \chi_{\{ \bm v : \mathcal{N}_t \ge M \}} d\lambda. \]
Since $\lambda \ll \text{Leb}$, the Radon--Nikodym theorem guarantees that there is a unique $h \in L^1(\Leb)$ with $\int_{S^{d-1}} h d\Leb = 1$ (as $\lambda$ is a probability measure) such that 
\[ \int_{S^{d-1}} \chi_{\{ \bm v : \mathcal{N}_t \ge M \}} d\lambda = \int_{S^{d-1}} \chi_{\{ \bm v : \mathcal{N}_t \ge M \}} h d\Leb. \]
We now fix a parameter $R$ (to be chosen later) and write: 
\[ \int_{S^{d-1}} \chi_{\{ \bm v : \mathcal{N}_t \ge M \}} h d\Leb = \int_{\{ h \le R \}}  \chi_{\{ \bm v : \mathcal{N}_t \ge M \}}  h d\Leb +  \int_{\{ h > R \}}  \chi_{\{ \bm v : \mathcal{N}_t \ge M \}}  h d\Leb. \]
By Markov's inequality and \eqref{convNt}, we get
\[ \int_{\{ h \le R \}}  \chi_{\{ \bm v : \mathcal{N}_t \ge M \}}  h d\Leb  \le \frac {R C_{\mathcal{A}}}M. \]
For the second summand, we have the bound (recall that $\int_{S^{d-1}} h d\Leb = 1$)
\[\int_{\{ h > R \}}  \chi_{\{ \bm v : \mathcal{N}_t \ge M \}}  h d\Leb \le \int_{\{ h > R \}}  h d\Leb \le \frac 1R. \]
It follows upon choosing $R = \frac 1{\sqrt{\delta}}$ that for every $\delta > 0$, there exists $M$ ($= \frac {C_{\mathcal{A}}}{\delta}$) such that for every $t > 0,$ \[\lambda(\{ \bm v \in S^{d-1} : \mathcal{N}_t > M \}) \le 2 \sqrt{\delta}, \] which proves tightness.
	We can thus, thanks to Prokhorov's criterion for relative compactness, find an increasing sequence $(t_i) \in (\mathbb{R}^+)^{\mathbb{N}}$ tending to infinity and a random variable $\mathcal{N}$ such that \[\mathcal{N}_{t_i} \xrightarrow[i \to \infty]{d} \mathcal{N}.\]
By assumption, for each $\varepsilon > 0$ we can also find a random variable $\mathcal{N}^\varepsilon$ such that \[ \mathcal{N}_t^\varepsilon \xrightarrow[t \to \infty]{d} \mathcal{N}^\varepsilon. \]
Now, for every $\varepsilon > 0$, every $m \in \mathbb{N}$ and every $i \in \mathbb{N}$, 
\begin{align*} 0 &\le \lambda(\{ \bm v \in S^{d-1} : \mathcal{N}_{t_i}^\varepsilon \ge m \}) - \lambda(\{ \bm v \in S^{d-1} : \mathcal{N}_{t_i} \ge m \}) \\
 &= \lambda(\{\mathcal{N}_{t_i}^\varepsilon \ge m \land \mathcal{N}_{t_i} \le m-1\}) \\
&\le \lambda(\{\mathcal{N}_{t_i}^\varepsilon - \mathcal{N}_{t_i} \ge 1\}) \\
&\le \sqrt{\varepsilon}  C_{\mathcal{A}} + \sqrt{\varepsilon}, \end{align*}
where the last inequality follows from a similar argument to the one used in the course of proving tightness --- choosing $R = \frac 1{\sqrt{\varepsilon}}$ this time.
Indeed, if we proceed as before to estimate $\int_{S^{d-1}} (\mathcal{N}_t^\varepsilon - \mathcal{N}_t) d\lambda$ and use \eqref{upperasydens} instead of \eqref{asydens} in the proof of the analogue of \autoref{boundedexpectation}, we end up with the upper bound $R \varepsilon C_\mathcal{A} + \frac 1R$ which is the sought upper bound with our choice of $R = \frac 1{\sqrt{\varepsilon}}$.

Letting $i \to \infty$, we get \[ \forall \varepsilon > 0, \, \forall m \in \mathbb{N}, \, 0 \le \lambda(\mathcal{N}^\varepsilon \ge m) - \lambda(\mathcal{N} \ge m) \le \sqrt{\varepsilon}  C_{\mathcal{A}} + \sqrt{\varepsilon}. \]
Therefore \[\forall m \in \mathbb{N}, \, \lambda(\mathcal{N} \ge m) = \lim_{\varepsilon \to 0} \lambda(\mathcal{N}^\varepsilon \ge m). \]
We obtain the same conclusion for every subsequence of $(\mathcal{N}_t)_{t >0}$ which converges in distribution, so $(\mathcal{N}_t)_{t >0}$ itself converges in distribution: for suppose the sequence does not converge in distribution to the common limit of all convergent subsequences, then there is a subsequence which does not converge in distribution to said common limit; by tightness, we obtain a subsubsequence which converges, necessarily to the common limit, hence the desired contradiction and conclusion.
\end{proof}

\section{Extension of the applications to certain weak model sets} \label{extension}
Recall that if $\mathcal{W}$ is such that $m_f(\partial W) = 0$, 
then we can use \autoref{mainthm} with $\mathcal{A} = \mathscr{C} \times \mathcal{W}$ where $\mathscr{C}$ is a suitably chosen Jordan measurable subset of $\mathbb{R}^d$ depending on the application (a cylinder for the free path length and a cone for directions).

It turns out we can also allow $m_f(\partial \mathcal{W})$ to be positive in some cases, such as the instances discussed in \autoref{examples}.

\begin{defn} Given $\varepsilon > 0$, define the window set $\mathcal{W} \subset \mathbb{A}_f^d$ to be $\varepsilon$-approximable if there exists a bounded set $\mathcal{W}^\varepsilon \subset \mathbb{A}_f^d$ satisfying the following conditions: 
 \begin{enumerate}
  \item $\mathcal{W} \subset \mathcal{W}^\varepsilon$
  \item $m_f(\partial W^\varepsilon) = 0$
  \item $m_f(\mathcal{W}^\varepsilon) \le m_f(\mathcal{W}) + \varepsilon$
 \end{enumerate}
\end{defn}

\begin{thm} If $\mathcal{P}(\mathcal{L}, \mathcal{W})$ is an adelic cut-and-project set with a positive asymptotic density and $\mathcal{W}$ is $\varepsilon$-approximable for every $\varepsilon > 0$, then \autoref{fpl} and \autoref{dir} hold. \end{thm}
\begin{proof} Thanks to the $\varepsilon$-approximability assumption we obtain, for each $\varepsilon > 0$, a set $\mathcal{W}^\varepsilon \subset \mathbb{A}_f^d$ whose boundary has measure zero and which has measure at most $\varepsilon$ bigger than that of $\mathcal{W}$. 
To conclude, we use the probabilistic argument from \autoref{densityincr} by taking $\mathcal{P}$ to be $\mathcal{P}(\mathcal{L},\mathcal{W})$, $\mathcal{P}^\varepsilon$ to be $\mathcal{P}(\mathcal{L},\mathcal{W^\varepsilon})$ and $\mathcal{A}$ to be the suitably chosen Jordan measurable set $\mathscr{C}$ mentioned at the beginning of this section. The key observation allowing us to apply \autoref{proba} is the following: 
\[ \mathscr{C} \Phi^{-t} K(\bm v) \cap \mathcal{P}(\mathcal{L}, \mathcal{W}^\varepsilon) = \emptyset \iff  \mathscr{C} \times \mathcal{W}^\varepsilon \cap \delta(\mathbb{Q})^d A \iota(K(\bm v)^{-1} \Phi^t) = \emptyset, \]
where the adelic lattice $\mathcal{L}$ is $\mathbb{Q}^d A$ for some $A \in \mathrm{SL}_d(\mathbb{A})$.
 \end{proof}

To echo our introductory discussion in \autoref{motiv}, we make the construction explicit in the case of the primitive lattice points by showing how the corresponding window set $\mathcal{W} = \prod_{p \in \mathbb{P}} (\mathbb{Z}_p^d \setminus p \mathbb{Z}_p^d)$ is $\varepsilon$-approximable for every $\varepsilon > 0$. The other cases discussed in \autoref{examples} are handled similarly.

Let $\varepsilon > 0$. We can find $n_\varepsilon$ such that \[ \prod_{i = 1}^{n_\varepsilon} \left( 1 - \frac 1{p_i^d} \right) \le \frac 1{\zeta(d)} + \varepsilon, \]
where $p_i$ is the $i$th prime.
Then we define $S_\varepsilon = \{2, 3, \ldots, p_{n_\varepsilon}\}$.
Now the set we use is \[ \mathcal{W}^\varepsilon = \prod_{p \in S_\varepsilon} (\mathbb{Z}_p^d \setminus p \mathbb{Z}_p^d) \times \prod_{p \notin S_\varepsilon} \mathbb{Z}_p^d. \]
Here we even have $\partial \mathcal{W}^\varepsilon = \emptyset$ since $\mathcal{W}^\varepsilon$ is both closed and open.
By construction, the Haar measures of $\mathcal{W}$ and $\mathcal{W}^\varepsilon$ are easily seen to satisfy the required conditions.

\section{Properties of the limiting distributions} \label{propslimdistr}

Using the explicit formulas for the limiting distributions obtained in \autoref{fpl} and \autoref{dir}, their continuity follows from a simple geometric argument using \autoref{adelicSiegel} to compute the expectation after applying Markov's inequality.

\begin{thm}\label{fplcont} For every $\bm \alpha \in \mathbb{R}^d$, $F_{\mathcal{P},\bm \alpha}$ (as given by the expressions \eqref{fpllimdistirr} and \eqref{fpllimdistrat}) is continuous on $\mathbb{R}^+$. \end{thm}
\begin{proof}
Recall the formula we obtained for the limiting distribution:

if  $\bm \alpha \in \mathbb{R}^d \setminus \mathbb{Q}^d$,
\[ F_{\mathcal{P},\bm \alpha}(\xi) = F_\mathcal{P}(\xi) = \mu(\{g \in \mathrm{ASL}_d(\mathbb{Q}) \backslash \mathrm{ASL}_d(\mathbb{A}) \, : \,  \delta(\mathbb{Q})^d g \cap (\mathscr{C}_\xi \times \mathcal{W}) = \emptyset \}) \]
while if $\bm \alpha \in \mathbb{Q}^d$, 
\[ F_\mathcal{P, \bm \alpha}(\xi) = \mu(\{M \in \mathrm{SL}_d(\mathbb{Q}) \backslash \mathrm{SL}_d(\mathbb{A}) \, : \,  \delta(\mathbb{Q})^d M \cap (\mathscr{C}_\xi \times \mathcal{W}_{\bm \beta}) = \emptyset \}),  \]
where $\bm \beta \in \delta(\mathbb{Q})^d$ is such that $\bm \alpha = \pi(\bm \beta)$ and $\mathcal{W}_{\bm \beta} = \mathcal{W} - \pi_{\text{int}}(\bm \beta)A_f$.

Consider first the case of $\bm \alpha \in \mathbb{R}^d \setminus \mathbb{Q}^d$.
Let $\xi \ge 0$. For $|h|$ small enough we have, following the outline of proof described before the statement of the theorem:
\begin{align*} |F_\mathcal{P}(\xi & + h) - F_\mathcal{P}(\xi)| \\ 
&\le \mu(\{ g \in \mathrm{ASL}_d(\mathbb{Q}) \backslash \mathrm{ASL}_d(\mathbb{A}) \, : \, \#(\delta(\mathbb{Q})^d g \cap ((\mathscr{C}_{\xi+h} \triangle \mathscr{C}_\xi) \times \mathcal{W}) \ge 1 \}) \\
&\le \int_{SL_d(\mathbb{A})} \int_{\mathbb{A}^d} \sum_{\bm q \in \delta(\mathbb{Q})^d} \chi_{(\mathscr{C}_{\xi+h} \triangle \mathscr{C}_\xi) \times \mathcal{W}}(\bm q M + \bm \zeta) d\bm \zeta d\mu_{\mathrm{SL}}(M) \\
&= \int_{\mathbb{A}^d} \chi_{(\mathscr{C}_{\xi+h} \triangle \mathscr{C}_\xi) \times \mathcal{W}}(\bm x) d\bm x \\
&= m((\mathscr{C}_{\xi+h} \triangle \mathscr{C}_\xi) \times \mathcal{W}) \\
&= \vol(\mathscr{C}_{\xi+h} \triangle \mathscr{C}_\xi) m_f(\mathcal{W}) \\
&= O(|h|),
\end{align*}
hence the desired continuity at $\xi$.

For $\bm \alpha \in \mathbb{Q}^d$ we obtain the same upper bound again, using \autoref{SiegelWeilformula} to compute the integral on the space of adelic lattices.

\end{proof}

\begin{thm} For every $\bm \alpha \in \mathbb{R}^d$ and every $c \in [0,1)$, $\sigma \mapsto G_{\mathcal{P}, \bm \alpha, c}(\sigma, \cdot)$ (as given by the expressions obtained in \autoref{dir})  is continuous on $\mathbb{R}^+$. \end{thm}
\begin{proof} This follows from the same argument as the proof of \autoref{fplcont} using the formula from \autoref{dir} as a starting point.
\end{proof}

We can also leverage results obtained by Marklof and Str\"ombergsson in the case of lattices in $\mathbb{R}^d$ to obtain a lower bound on the tail of the limiting distribution for the free path lengths.
Indeed, by choosing windows inside $\widehat{\mathbb{Z}}^d$ within our adelic cut-and-project scheme, we produce subsets of lattices in $\mathbb{R}^d$.

\begin{thm} For $\bm \alpha \in \mathbb{R}^d \setminus \mathbb{Q}^d$, we have
\begin{equation} F_\mathcal{P}(\xi) \ge \frac {\pi^{\frac {d-1}2}}{2^d d \Gamma(\frac {d+3}2) \zeta(d)} \xi^{-1} + O(\xi^{-1-\frac 2d}) \end{equation}
as $\xi \to \infty$.
 \end{thm}
\begin{proof}
We denote the limiting free path length distribution for lattices in $\mathbb{R}^d$ with generic initial condition obtained by Marklof and Str\"ombergsson in \cite[Corollary 4.1]{marklof_strombergsson_free_path_length_2010} by $\overline{F}$. Because our adelic model sets are subsets of lattices in $\mathbb{R}^d$ we immediately get the lower bound $F_\mathcal{P} \ge \overline{F}$.
The claim now follows from \cite[Theorem 1.13]{MarklofStrombergssonAsy}, which is equivalent to the assertion that 
\[\overline{F}(\xi) = \frac {\pi^{\frac {d-1}2}}{2^d d \Gamma(\frac {d+3}2) \zeta(d)} \xi^{-1} + O(\xi^{-1-\frac 2d}) \]
as $\xi \to \infty$.
\end{proof}

\bibliographystyle{plain}
\bibliography{bibliography}

\footnotesize
\parindent=0pt

\textsc{Daniel El-Baz, School of Mathematical Sciences, Tel Aviv University, Tel Aviv, Israel} \texttt{danielelbaz@mail.tau.ac.il}

\end{document}